\documentclass{amsart}
\usepackage{amsrefs}
\usepackage{amssymb}
\usepackage[usesnames,svgnames]{xcolor}
\usepackage[sc]{mathpazo}
\usepackage[T1]{fontenc}
\usepackage{tikz,pgf}
\usepackage{etoolbox}
\usepackage{ifthen}
\usepackage[colorlinks=true,
	urlcolor=MidnightBlue,
	citecolor=DarkGreen]{hyperref}
\usepackage{subfigure}
\usepackage{enumerate}
\usepackage{todonotes}
\pgfdeclarelayer{background}
\pgfsetlayers{background,main}
\newtheorem{theorem}{Theorem}[section]
\newtheorem{proposition}[theorem]{Proposition}

\newtheorem{lemma}[theorem]{Lemma}
\newtheorem{corollary}[theorem]{Corollary}

\theoremstyle{remark}
\newtheorem{remark}[theorem]{Remark}
\newtheorem{example}[theorem]{Example}
\newcommand{\alert}[1]{{\color{DarkGreen}\emph{#1}}}
\newcommand{\ie}{\text{i.e.}\;}

\newcommand{\st}{^{\text{st}}}
\newcommand{\nd}{^{\text{nd}}}
\renewcommand{\th}{^{\text{th}}}
\newcommand{\BB}{\mathcal{B}}
\newcommand{\DD}{\mathcal{D}}
\newcommand{\II}{\mathcal{I}}
\newcommand{\KK}{\mathcal{K}}
\newcommand{\LL}{\mathcal{L}}

\newcommand{\pf}{\mathfrak{p}}
\newcommand{\qf}{\mathfrak{q}}
\newcommand{\hh}{\mathbf{h}}
\newcommand{\cc}{\mathsf{c}}

\newcommand{\pathThree}[4]{
	\begin{tikzpicture}[scale=#3]\tiny
		\draw[white!50!gray](0,0) grid[step=.4] (1.2,1.2);
		\draw(.6,1.3) node{};
		\ifstrequal{#4}{A}{\draw(0,0) -- (1.2,1.2);}{}
		\begin{pgfonlayer}{background}
			\fill[#2] (0,0) -- (1.2,0) -- (1.2,1.2) -- (0,1.2) -- cycle;
		\end{pgfonlayer}
		\foreach \a/\b/\c in {#1}{
			\draw(0,0) -- (0,.4*\a) -- (.4,.4*\a) -- (.4,.4*\b) -- (.8,.4*\b) -- (.8,.4*\c) -- (1.2,.4*\c) -- (1.2,1.2);
		}
	\end{tikzpicture}
}
\newcommand{\pathFour}[4]{
	\begin{tikzpicture}[scale=#3]\tiny
		\draw[white!50!gray](0,0) grid[step=.4] (1.6,1.6);
		\draw(.8,1.7) node{};
		\ifstrequal{#4}{A}{\draw(0,0) -- (1.6,1.6);}{}
		\begin{pgfonlayer}{background}
			\fill[#2] (0,0) -- (1.6,0) -- (1.6,1.6) -- (0,1.6) -- cycle;
		\end{pgfonlayer}
		\foreach \a/\b/\c/\d in {#1}{
			\draw(0,0) -- (0,.4*\a) -- (.4,.4*\a) -- (.4,.4*\b) -- (.8,.4*\b) -- (.8,.4*\c) -- (1.2,.4*\c) -- (1.2,.4*\d) -- (1.6,.4*\d) -- (1.6,1.6);
		}
	\end{tikzpicture}
}

\newcommand{\dyckBThree}[3]{
	\begin{tikzpicture}[scale=#3]\tiny
		\draw[white!50!gray](0,0) grid[step=.4] (1.2,2.4);
		\draw(0,0) -- (1.2,1.2) -- (0,2.4);
		\draw(.6,2.5) node{};
		\begin{pgfonlayer}{background}
			\fill[#2] (0,0) -- (1.2,0) -- (1.2,2.4) -- (0,2.4) -- cycle;
		\end{pgfonlayer}
 		\foreach \ax/\ay/\bx/\by/\cx/\cy/\dx/\dy/\ex/\ey/\fx/\fy in {#1}{
			  \draw(0,0) -- (0,\ay) -- (\ax,\ay) -- (\ax,\by) -- (\bx,\by) -- (\bx,\cy) -- (\cx,\cy) 
 			  -- (\cx,\dy) -- (\dx,\dy) -- (\dx,\ey) -- (\ex,\ey) -- (\ex,\fy) -- (\fx,\fy);

		}
	\end{tikzpicture}
}
\newcommand{\cOne}{white}
\newcommand{\cTwo}{white!80!green}

\author{Henri M{\"u}hle}
\address{LIX, {\'E}cole Polytechnique, F-91128 Palaiseau, France}
\thanks{This work was funded by the FWF Research Grant No. Z130-N13, as well as a Public Grant overseen by the French National Research Agency (ANR) as part of the ``Investissements d'Avenir'' Program (Reference: ANR-10-LABX-0098), and Digiteo project PAAGT (Nr. 2015-3161D)}
\email{henri.muehle@lix.polytechnique.fr}
\title{A Heyting Algebra on Dyck Paths of Type $A$ and $B$}
\begin{document}

\maketitle

\begin{abstract}
	In this article we investigate the lattices of Dyck paths of type $A$ and $B$ under dominance order, and explicitly describe their Heyting algebra structure.  This means that each Dyck path of either type has a relative pseudocomplement with respect to some other Dyck path of the same type.  While the proof that this lattice forms a Heyting algebra is quite straightforward, the explicit computation of the relative pseudocomplements using the lattice-theoretic definition is quite tedious.  We give a combinatorial description of the Heyting algebra operations join, meet, and relative pseudocomplement in terms of height sequences, and we use these results to derive formulas for pseudocomplements and to characterize the regular elements in these lattices.
\end{abstract}

\section{Introduction}
	\label{sec:introduction}
In this article we mostly consider lattice paths from $(0,0)$ to $(n,n)$ that consist only of up- and right-steps and that stay (weakly) above the diagonal $x=y$.  We denote the set of all these paths by $D_{n}^{A}$, and refer to them as \alert{Dyck paths of type $A$}, a notation that is justified in the next paragraphs.

A path $\pf\in D_{n}^{A}$ \alert{dominates} another Dyck path $\pf'$ of the same length, if $\pf'$ stays weakly below $\pf$ at all time, and in that case we write $\pf'\leq_{D}\pf$.  We write $\DD_{n}^{A}$ for the resulting poset.  In fact, $\DD_{n}^{A}$ is a distributive lattice, and probably first appeared in \cite{stanley75fibonacci}*{Example~4} as a partial order on the set of order ideals of a triangular poset with $n-1$ minimal elements.  (It is an easy exercise to work out the isomorphism between these two posets.)  Apart from these two guises, the lattice $\DD_{n}^{A}$ also appears as the Bruhat order on noncrossing partitions~\cite{gobet16noncrossing} and $312$-avoiding permutations~\cites{armstrong09generalized,barcucci05distributive}.  A lot of research has been done on enumerative and structural aspects of $\DD_{n}^{A}$~\cites{ferrari05lattices,ferrari11lattices,ferrari13unimodality,ferrari14edges,ferrari15chains}, and it has interesting connections to the change of basis matrix of the Temperley--Lieb algebra~\cites{cautis03matrix,gobet16temperley}.  

The lattice $\DD_{n}^{A}$ naturally constitutes a principal order filter in the lattice $\LL(n,n)$ of all lattice paths from $(0,0)$ to $(n,n)$ using only up- and right-steps.  In general, $\LL(n,m)$ is also a distributive lattice, and it is clearly a sublattice of Young's lattice.  For a certain choice of parameters $\LL(n,m)$ is isomorphic to the Bruhat order on parabolic quotients of the symmetric group~\cite{stanley80weyl}*{Section~4}.  These lattices have been further studied in \cites{bennett94two,proctor84bruhat} and the references given therein.

The previous paragraphs suggest a strong connection between $\DD_{n}^{A}$ and the symmetric group $\mathfrak{S}_{n}$, which is isomorphic to the Coxeter group $A_{n-1}$.  This is the main motivation for the superscript ``$A$''.  We also remark that the triangular poset with $n-1$ minimal elements also appears naturally in this context, namely as the so-called root poset of $A_{n-1}$.  

It is well known that the cardinality of $D_{n}^{A}$ is given by the $n\th$ Catalan number $\tfrac{1}{n+1}\tbinom{2n}{n}$~\cite{stanley01enumerative}*{Exercise~6.19(i)}.  It is often the case for combinatorial objects associated with the symmetric group and enumerated by the Catalan numbers, that the subset of these objects with a central symmetry is an interesting combinatorial family in its own right.  Famous examples are noncrossing partitions fixed under a half turn, or centrally symmetric triangulations of a convex polygon.  

Something similar happens for Dyck paths of type $A$.  Consider the set $D_{n}^{B}$ of lattice paths from $(0,0)$ to $(2n,2n)$ that stay weakly above the diagonal $x=y$, and that are invariant under reflection about the diagonal $x=2n-y$.  For brevity, identify each of these paths with the subpath consisting of the first $2n$ steps.  The cardinality of $D_{n}^{B}$ is given by the central binomial coefficient $\binom{2n}{n}$, which is also the number of order ideals in the root poset of the Coxeter group $B_{n}$; an explicit bijection was given in \cite{stump13more}*{Section~3.1}.  Consequently, we call the elements of $D_{n}^{B}$ \alert{Dyck paths of type $B$}.  We remark that the Coxeter group $B_{n}$ is isomorphic to the hyperoctahedral group of rank $n$.  It is straightforward to show that the dominance order on $D_{n}^{B}$ forms a distributive lattice, denoted by $\DD_{n}^{B}$, which is isomorphic to the lattice of order ideals of the root poset of the Coxeter group $B_{n}$. 

The main purpose of this article is to outline further structural commonalities between the lattices $\DD_{n}^{A}$ and $\DD_{n}^{B}$, which fit nicely into the combinatorial relationship between the Coxeter groups $A_{n-1}$ and $B_{n}$ that is part of the stream of Coxeter-Catalan combinatorics.  In particular, since $\DD_{n}^{A}$ and $\DD_{n}^{B}$ are both finite distributive lattices, they naturally possess a Heyting algebra structure, and it is our goal to combinatorially understand the relation between these Heyting algebras.  Our first main result is the following theorem.

\begin{theorem}\label{thm:dyck_algebras}
	For $n>0$ the lattice $\DD_{n}^{B}$ of Dyck paths of type $B$ under dominance order forms a Heyting algebra.  The sublattice $\DD_{n}^{A}$ of Dyck paths of type $A$ under dominance order forms a Heyting algebra as well, but it is not a Heyting subalgebra of the former.  Conversely, however, $\DD_{n}^{B}$ is (isomorphic to) a Heyting subalgebra of $\DD_{2n}^{A}$.
\end{theorem}

The second main contribution of this article, is a purely combinatorial description of these Heyting algebras, \ie we give explicit formulas for join, meet, and relative pseudocomplements in these Heyting algebras using only the combinatorial realization of Dyck paths in terms of height sequences.  Moreover, we use these formulas to describe pseudocomplements and characterize the regular elements in these algebras.  A thorough investigation of the Heyting algebra of Dyck paths of type $A$ from a logic-theoretical standpoint was recently carried out in \cite{ferrari15dyck}.  More precisely, in \cite{ferrari15dyck} the Heyting algebra of Dyck paths of type $A$ was related to a fragment of interval temporal logic.  

\smallskip

This article is organized as follows: in Section~\ref{sec:preliminaries} we recall the necessary lattice-theoretic notions of Heyting algebra and distributive lattice.  Moreover, we formally define Dyck paths, realize them in terms of height sequences, and relate Dyck paths of type $B$ to centrally symmetric Dyck paths of type $A$.  Subsequently we characterize the join-prime Dyck paths, and relate the lattices of Dyck paths to the lattices of order ideals of some triangular posets.  In Section~\ref{sec:formulas}, we explicitly describe the height functions of the relative pseudocomplements and pseudocomplements in the Heyting algebras of Dyck paths.  More precisely, we start with the investigation of the Heyting algebra of \emph{all} monotone lattice paths, and derive the formulas in the Heyting algebras of Dyck paths of type $A$ and $B$ from this.  

\section{Preliminaries}
	\label{sec:preliminaries}
In this section we define the basic notions needed in this article.  Throughout this article we use the abbreviation $[n]=\{1,2,\ldots,n\}$.  

\subsection{Heyting Algebras and Distributive Lattices}
	\label{sec:heyting_distributive}
We start by recalling the notions of Heyting algebra and distributive lattice.  For further background we refer the reader to \cite{balbes74distributive} or \cite{blyth05lattices}*{Chapter~7}.  

Let $\LL=(L,\leq)$ be a lattice with least element $\hat{0}$ and greatest element $\hat{1}$.  Given $x,y\in L$ we say that the greatest element $z\in L$ satisfying 
\begin{equation}\label{eq:pseudocomplement}
	x\wedge z\leq y
\end{equation}
is (if it exists) the \alert{relative pseudocomplement} of $x$ with repect to $y$, and we usually write $x\to y$.  If relative pseudocomplements exist for all $x,y\in L$, then $\LL$ is a \alert{Heyting algebra}.  Moreover, if $\LL$ is a Heyting algebra, then the element $x\to\hat{0}$ for $x\in L$ is the \alert{pseudocomplement} of $x$, and we usually write $x^{\cc}$.  An element $x\in L$ is \alert{regular} if $(x^{\cc})^{\cc}=x$.  It is straightforward to verify that the poset $\BB=(B,\leq)$, where $B=\{x\in L\mid x\;\text{is regular}\}$, is a Boolean lattice.

\begin{lemma}[\cite{balbes74distributive}*{Theorem~IX.1.3(iii)}]\label{lem:pseudocomplement_comparable}
	Let $\LL=(L,\leq)$ be a Heyting algebra.  If $x,y\in L$ satisfy $x\leq y$, then $x\to y=\hat{1}$.  
\end{lemma}

If every three elements $x,y,z\in L$ satisfy one of the two equivalent laws
\begin{align}\label{eq:meet_distributive}
	x\wedge(y\vee z) & = (x\wedge y)\vee(x\wedge z),\quad\text{and}\\
	x\vee(y\wedge z) & = (x\vee y)\wedge(x\vee z),
\end{align}
then $\LL$ is \alert{distributive}.  The following theorem states the connection between distributive lattices and Heyting algebras.

\begin{theorem}[\cite{blyth05lattices}*{Theorem~7.10}]\label{thm:distributive_heyting}
	Every Heyting algebra is distributive.  Conversely, every finite distributive lattice is a Heyting algebra.
\end{theorem}

We recall some further background on Heyting algebras, borrowing slightly from category theory.  Let $\KK,\LL$ be two Heyting algebras, and let $f,g:\KK\to\LL$ be two Heyting algebra morphisms.  Define $\text{Eq}(f,g)=\{k\in K\mid f(k)=g(k)\}$ to be the \alert{equalizer} of $f$ and $g$.  We have the following result.

\begin{proposition}\label{prop:equalizer_subobject}
	Let $\KK,\LL$ be Heyting algebras, and let $f,g:\KK\to\LL$ be two Heyting algebra morphisms.  The equalizer $\text{Eq}(f,g)$ is a Heyting subalgebra of $\KK$.
\end{proposition}
\begin{proof}
	This follows for instance from \cite{barr98category}*{Proposition~9.1.5}.
\end{proof}

\subsection{Dyck Paths of Type $A$ and $B$}
	\label{sec:dyck_paths}
A \alert{Dyck path of semilength $n$} is a lattice path on $\mathbb{N}^{2}$ which starts at $(0,0)$, which consists of $2n$ steps either of the form $(0,1)$ (so-called \alert{up-steps}) or of the form $(1,0)$ (so-called \alert{right-steps}), and which stays weakly above the diagonal $x=y$.  A Dyck path of semilength $n$ \alert{is of type $A$} if it ends at $(n,n)$.  If we do not pose the restriction on the Dyck path enforcing it to end at $(n,n)$, then it \alert{is of type $B$}.  Let $D_{n}^{A}$ denote the set of Dyck paths of semilength $n$ being of type $A$, and let $D_{n}^{B}$ denote the set of Dyck paths of semilength $n$ being of type $B$.  Clearly we have $D_{n}^{A}\subseteq D_{n}^{B}$.  It is well known that
\begin{displaymath}
	\bigl\lvert D_{n}^{A}\bigr\rvert = \frac{1}{n+1}\binom{2n}{n},\quad\text{and}\quad\bigl\lvert D_{n}^{B}\bigr\rvert = \binom{2n}{n},
\end{displaymath}
and these numbers are known as \alert{Catalan numbers of type $A$ and $B$}, respectively.  See for instance \cite{stanley01enumerative}*{Exercise~6.19(i)} and \cite{nilsson12enumeration}*{Corollary~6}.  

It is standard, see for instance \cite{deutsch99dyck}*{Section~2} or \cite{nilsson12enumeration}*{Section~3.2}, that a Dyck path $\pf\in D_{n}^{B}$ can be encoded by a \alert{Dyck word}, namely a word $w_{\pf}$ of length $2n$ over the alphabet $\{u,r\}$ in which every prefix contains at least as many $u$'s as it contains $r$'s.  If $\pf\in D_{n}^{A}$, then $w_{\pf}$ is additionally required to contain exactly $n$ times the letter $u$ and $n$ times the letter $r$.  See Figure~\ref{fig:dyck_words} for an illustration.

\begin{figure}
	\centering
	\begin{tikzpicture}\small
		\draw(1,.2) node{\begin{tikzpicture}
				\def\x{1};
				\def\y{1};
				\draw[white!80!gray](0,0) grid[step=.4] (4.4,4.4);
				\draw(0,0) -- (4.4,4.4);
				\draw(0,0) -- (0,1.2) -- (.4,1.2) -- (.4,2) -- (.8,2) -- (.8,2.8) -- (2,2.8) -- (2,3.2) -- (2.4,3.2) -- (2.4,4) -- (2.8,4) -- (2.8,4.4) -- (4.4,4.4);
%				\draw(.6,3.8) node{\normalsize $\pf$};
				\draw(2.4,-.6) node{\begin{tabular}{l}
					$\hh = (3,5,7,7,7,8,10,11,11,11,11)$\\
					$w = uuuruuruurrruruururrrr$\\
					\end{tabular}};
			\end{tikzpicture}};
		
		\draw(7,1) node{\begin{tikzpicture}
				\def\x{1};
				\def\y{1};
				\draw[white!80!gray](0,0) grid[step=.4] (4.4,6);
				\draw(0,0) -- (4.4,4.4) -- (2.8,6);
				\draw(0,0) -- (0,1.2) -- (.4,1.2) -- (.4,2) -- (.8,2) -- (.8,2.8) -- (1.6,2.8) -- (1.6,4) -- (2,4) -- (2,5.6) -- (2.8,5.6) -- (2.8,6);
%				\draw(.6,3.8) node{\normalsize $\pf'$};
				\draw(2.2,-.6) node{\begin{tabular}{l}
					$\hh = (3,5,7,7,10,14,14,15)$\\
					$w = uuuruuruurruuuruuuurru$\\
					\end{tabular}};
			\end{tikzpicture}};
	\end{tikzpicture}
	\caption{Two Dyck paths of semilength $11$, and their corresponding height sequences and Dyck words.  The left path is of type $A$, and the right path is of type $B$.}
	\label{fig:dyck_words}
\end{figure}
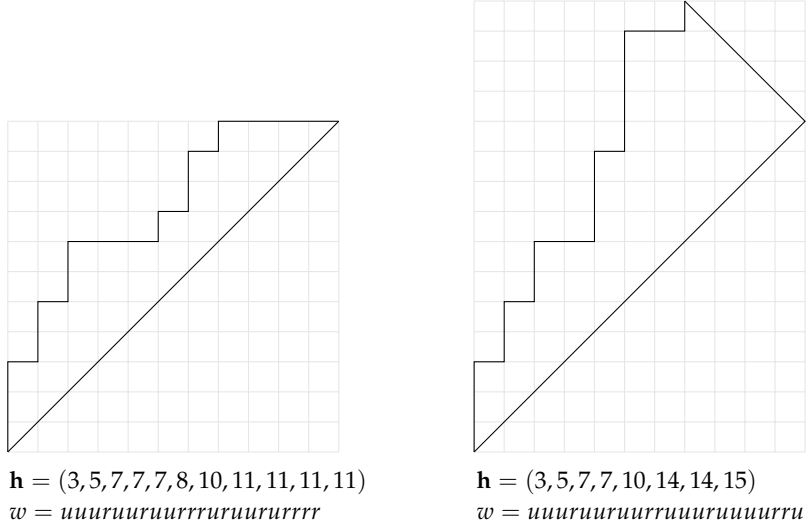

\subsubsection{Type $A$}
	\label{sec:dyck_paths_a}
For $\pf\in D_{n}^{A}$ define a sequence $\hh_{\pf}=(h_{1},h_{2},\ldots,h_{n})$, where $h_{i}$ is the number of $u$'s occurring in $w_{\pf}$ before the $i\th$ occurence of the letter $r$, and call this sequence the \alert{height sequence} of $\pf$.  The next lemma implies that the entry $h_{i}$ in $\hh_{\pf}$ determines precisely the height of $\pf$ at abscissa $i-1$. 

\begin{lemma}\label{lem:height_sequence_a}
	If $\pf\in D_{n}^{A}$, then $\hh_{\pf}=(h_{1},h_{2},\ldots,h_{n})$ satisfies $h_{1}\leq h_{2}\leq\cdots\leq h_{n}$, and $i\leq h_{i}\leq n$ for all $i\in[n]$.  Conversely, each such sequence uniquely determines a Dyck path in $D_{n}^{A}$.  
\end{lemma}
\begin{proof}
	If $\pf$ is a Dyck path of type $A$, then the conditions $h_{1}\leq h_{2}\leq\cdots\leq h_{n}=n$ and $i\leq h_{i}$ for $i\in[n]$ both follow easily from the assumption that $w_{\pf}$ is a Dyck word.
	
	\smallskip
	
	Conversely, if $\hh=(h_{1},h_{2},\ldots,h_{n})$ has the desired properties, then it is quickly checked that 
	\begin{displaymath}
		w_{\hh} = \underbrace{uu\cdots u}_{h_{1}}r\underbrace{uu\cdots u}_{h_{2}-h_{1}}r\cdots r\underbrace{uu\cdots u}_{h_{n}-h_{n-1}}r.
	\end{displaymath}
	is a Dyck word of type $A$. 
\end{proof}

\begin{remark}
	Our notion of a height sequence associated with a Dyck path of type $A$ coincides with the max-vector of a noncrossing partition defined in \cite{barcucci05distributive}*{Section~4}.
\end{remark}

Let $\pf,\pf'\in D_{n}^{A}$ with associated height sequences $\hh_{\pf}=(h_{1},h_{2},\ldots,h_{n})$ and $\hh_{\pf'}=(h'_{1},h'_{2},\ldots,h'_{n})$.  Define $\pf\leq_{D}\pf'$ if and only if $h_{i}\leq h'_{i}$ for $i\in[n]$, and call this partial order the \alert{dominance order on $ D_{n}^{A}$}.  We usually write $\DD_{n}^{A}$ for the poset $\bigl(D_{n}^{A},\leq_{D}\bigr)$.  Figure~\ref{fig:dominance_a4} shows $\DD_{4}^{A}$.  

\begin{lemma}\label{lem:dominance_words_a}
	For any $\pf,\pf'\in D_{n}^{A}$ we have $\pf\leq_{D}\pf'$ if and only if in every prefix of $w_{\pf}$ there are at least as many $r$'s as there are in the prefix of $w_{\pf'}$ of the same length.
\end{lemma}
\begin{proof}
	This is a straightforward computation.
\end{proof}

\begin{theorem}[\cite{ferrari05lattices}*{Corollary~2.2}]\label{thm:distributive_a}
	For $n>0$ the poset $\DD_{n}^{A}$ is a distributive lattice.
\end{theorem}
%\begin{proof}
%	It is straightforward to verify that for any two Dyck paths $\pf,\pf'\in D_{n}^{A}$, whose height sequences are $\hh_{\pf}=(h_{1},h_{2},\ldots,h_{n})$ and $\hh_{\pf'}=(h'_{1},h'_{2},\ldots,h'_{n})$, respectively, their meet can be defined via the height sequence 
%	\begin{displaymath}
%		\hh_{\pf\wedge_{D}\pf'} = \bigl(\min\{h_{1},h'_{1}\},\min\{h_{2},h'_{2}\},\ldots,\min\{h_{n},h'_{n}\}\bigr),
%	\end{displaymath}
%and their join can be defined via the height sequence 
%	\begin{displaymath}
%		\hh_{\pf\vee_{D}\pf'} = \bigl(\max\{h_{1},h'_{1}\},\max\{h_{2},h'_{2}\},\ldots,\max\{h_{n},h'_{n}\}\bigr).
%	\end{displaymath}
%	Since $\min$ and $\max$ are distributive, the result follows.
%\end{proof}

\begin{figure}
	\centering
	\begin{tikzpicture}\small
		\def\x{1.5};
		\def\y{1.15};
		\draw(2*\x,1*\y) node(v1){\pathFour{1/2/3/4}{\cTwo}{.4}{A}};
		\draw(1*\x,2*\y) node(v2){\pathFour{2/2/3/4}{\cTwo}{.4}{A}};
		\draw(2*\x,2*\y) node(v3){\pathFour{1/3/3/4}{\cTwo}{.4}{A}};
		\draw(3*\x,2*\y) node(v4){\pathFour{1/2/4/4}{\cTwo}{.4}{A}};
		\draw(1*\x,3*\y) node(v5){\pathFour{2/3/3/4}{\cOne}{.4}{A}};
		\draw(2*\x,3*\y) node(v6){\pathFour{2/2/4/4}{\cTwo}{.4}{A}};
		\draw(3*\x,3*\y) node(v7){\pathFour{1/3/4/4}{\cOne}{.4}{A}};
		\draw(1*\x,4*\y) node(v8){\pathFour{3/3/3/4}{\cTwo}{.4}{A}};
		\draw(2*\x,4*\y) node(v9){\pathFour{2/3/4/4}{\cOne}{.4}{A}};
		\draw(3*\x,4*\y) node(v10){\pathFour{1/4/4/4}{\cTwo}{.4}{A}};
		\draw(1.5*\x,5*\y) node(v11){\pathFour{3/3/4/4}{\cOne}{.4}{A}};
		\draw(2.5*\x,5*\y) node(v12){\pathFour{2/4/4/4}{\cOne}{.4}{A}};
		\draw(2*\x,6*\y) node(v13){\pathFour{3/4/4/4}{\cOne}{.4}{A}};
		\draw(2*\x,7*\y) node(v14){\pathFour{4/4/4/4}{\cTwo}{.4}{A}};
		\draw(v1) -- (v2);
		\draw(v1) -- (v3);
		\draw(v1) -- (v4);
		\draw(v2) -- (v5);
		\draw(v2) -- (v6);
		\draw(v3) -- (v5);
		\draw(v3) -- (v7);
		\draw(v4) -- (v6);
		\draw(v4) -- (v7);
		\draw(v5) -- (v9);
		\draw(v6) -- (v9);
		\draw(v7) -- (v9);
		\draw(v5) -- (v8);
		\draw(v7) -- (v10);
		\draw(v8) -- (v11);
		\draw(v9) -- (v11);
		\draw(v9) -- (v12);
		\draw(v10) -- (v12);
		\draw(v11) -- (v13);
		\draw(v12) -- (v13);
		\draw(v13) -- (v14);
	\end{tikzpicture}
	\caption{The lattice $\DD_{4}^{A}$.  The highlighted paths are regular.}
	\label{fig:dominance_a4}
\end{figure}

The set $D_{n}^{A}$ comes naturally equipped with a nontrivial automorphism $\psi$, the \alert{reflection map}, which corresponds to a reflection of the lattice path about the diagonal $y=n-x$.  In terms of Dyck words, this automorphism can be expressed as follows.  Let $\pf\in D_{n}^{A}$ have associated Dyck word $w_{\pf}$.  We construct a new word $w'$ from $w_{\pf}$ by simultaneously replacing each $u$ by an $r$, and each $r$ by an $u$.  Subsequently, we construct a word $w''$ from $w'$, by sending the $i\th$ letter of $w'_{\pf}$ to the $(n-i+1)\st$ letter of $w''_{\pf}$.  It is straightforward to check that $w''$ is the Dyck word of some path $\psi(\pf)\in D_{n}^{A}$.  See Figure~\ref{fig:automorphism_a} for an illustration.

%\begin{lemma}\label{lem:involution_a_welldefined}
%	The previously described involution $\psi$ is well defined, \ie for any $\pf\in D_{n}^{A}$ we have $\psi(\pf)\in D_{n}^{A}$. 
%\end{lemma}
%\begin{proof}
%	Let $\pf\in D_{n}^{A}$, and let $w''$ be the word constructed as above.  It is immediate that $w''$ has $2n$ letters, precisely $n$ of which are $u$'s and $n$ of which are $r$'s.  Consider the prefix of $w''$ that ends in the $i\th$ occurrence of the letter $r$, and suppose that this prefix has length $i+c$, where $i$ is the number of $r$'s in this prefix and $c$ is the number of $u$'s in this prefix.  
%	
%	This prefix corresponds to the suffix of $w_{\pf}$ which starts at the $(n-i-c+1)\st$ position, and the first letter of this suffix is a $u$.  Moreover, this suffix contains $i$ times the letter $u$, and $c$ times the letter $r$.  The prefix of $w_{\pf}$ consisting of the first $n-i-c$ letters consequently contains $n-i$ times the letter $u$, and $n-c$ times the letter $r$.  Since $w_{\pf}$ is a Dyck word, we conclude $n-i\geq n-c$, which implies $c\geq i$.  Hence in the prefix of $w''$ that we originally considered the letter $u$ occurs at least as often as the letter $r$, which implies that $w''$ is a Dyck word of type $A$. 
%	
%	The fact that $\psi$ is an involution is immediate from the definition, since both switching letters, and reversing their position are involutions.
%\end{proof}

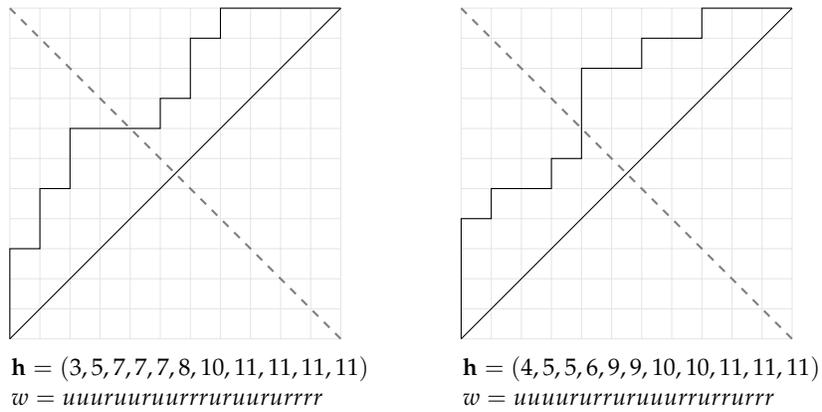
\begin{figure}
	\centering
	\begin{tikzpicture}\small
		\draw(1,1) node{\begin{tikzpicture}
				\def\x{1};
				\def\y{1};
				\draw[white!80!gray](0,0) grid[step=.4] (4.4,4.4);
				\draw[gray,dashed,thick](0,4.4) -- (4.4,0);
				\draw(0,0) -- (4.4,4.4);
				\draw(0,0) -- (0,1.2) -- (.4,1.2) -- (.4,2) -- (.8,2) -- (.8,2.8) -- (2,2.8) -- (2,3.2) -- (2.4,3.2) -- (2.4,4) -- (2.8,4) -- (2.8,4.4) -- (4.4,4.4);
%				\draw(.6,3.8) node{\normalsize $\pf$};
				\draw(2.4,-.6) node{\begin{tabular}{l}
					$\hh = (3,5,7,7,7,8,10,11,11,11,11)$\\
					$w = uuuruuruurrruruururrrr$\\
					\end{tabular}};
			\end{tikzpicture}};
		
		\draw(7,1) node{\begin{tikzpicture}
				\def\x{1};
				\def\y{1};
				\draw[white!80!gray](0,0) grid[step=.4] (4.4,4.4);
				\draw[gray,dashed,thick](0,4.4) -- (4.4,0);
				\draw(0,0) -- (4.4,4.4);
				\draw(0,0) -- (0,1.6) -- (.4,1.6) -- (.4,2) -- (1.2,2) -- (1.2,2.4) -- (1.6,2.4) -- (1.6,3.6) -- (2.4,3.6) -- (2.4,4) -- (3.2,4) -- (3.2,4.4) -- (4.4,4.4);
%				\draw(.6,3.8) node{\normalsize $\psi(\pf)$};
				\draw(2.4,-.6) node{\begin{tabular}{l}
					$\hh = (4,5,5,6,9,9,10,10,11,11,11)$\\
					$w = uuuururruruuurrurrurrr$\\
					\end{tabular}};
			\end{tikzpicture}};
	\end{tikzpicture}
	\caption{Illustration of the reflection map $\psi$.  The right path is the image of the left path under $\psi$.}
	\label{fig:automorphism_a}
\end{figure}

\begin{lemma}\label{lem:involution_a_height}
	Let $\pf\in D_{n}^{A}$ have height sequence $\hh_{\pf}=(h_{1},h_{2},\ldots,h_{n})$.  The Dyck path $\psi(\pf)$ has height sequence $\hh_{\psi(\pf)}=(h'_{1},h'_{2},\ldots,h'_{n})$, given by
	\begin{displaymath}
		h'_{n-h_{i}+1} = h'_{n-h_{i}+2} = \cdots = h_{n-h_{i-1}} = n-i+1,
	\end{displaymath}
	for any $i\in[n]$ with $h_{i}>h_{i-1}$, and where we set $h_{0}=0$.  
\end{lemma}
\begin{proof}
	Let $\pf\in D_{n}^{A}$ have height sequence $\hh_{\pf}=(h_{1},h_{2},\ldots,h_{n})$, and pick $i\in[n]$ with $h_{i}>h_{i-1}$.  More precisely, say that $h_{i}-h_{i-1}=c>0$.  By definition this means that in $w_{\pf}$ there are exactly $c$-many letters $u$ between the $(i-1)\st$ and the $i\th$ occurrence of the letter $r$.  More precisely, these are the $(h_{i-1}+1)\st,(h_{i-1}+2)\nd,\ldots,h_{i}\th$ letters $u$ in $w_{\pf}$.
	
	Hence, by construction, in $w_{\psi(\pf)}$ occur precisely $c$ consecutive $r$'s between the $(n-i+1)\st$ and $(n-i+2)\nd$ occurrence of the letter $u$.  (If $i=1$, then the previous is to be read as ``after the $n\th$ occurrence of the letter $u$'').  More precisely, these are the $(n-h_{i}+1)\st,(n-h_{i}+2)\nd,\ldots,(n-h_{i-1})\th$ letters $r$ in $w_{\psi(\pf)}$, which implies that $h'_{n-h_{i}+1}=h'_{n-h_{i}+2}=\cdots=h'_{n-h_{i-1}}=n-i+1$, which concludes the proof.
\end{proof}

\begin{lemma}\label{lem:involution_a_automorphism}
	The map $\psi$ is a lattice automorphism of $\DD_{n}^{A}$, \ie $\pf,\pf'\in D_{n}^{A}$ satisfy $\pf\leq_{D}\pf'$ if and only if $\psi(\pf)\leq_{D}\psi(\pf')$.  	
\end{lemma}
\begin{proof}
	Suppose that $\pf\leq_{D}\pf'$.  Lemma~\ref{lem:dominance_words_a} implies that in every prefix of $\pf$ the number of $r$'s is at least as big as the number of $r$'s in the prefix of the same length in $w_{\pf'}$.  This implies by construction that the same is true for $\psi(\pf)$ and $\psi(\pf')$, which yields $\psi(\pf)\leq_{D}\psi(\pf')$.  Since $\psi$ is an involution, the converse follows.
\end{proof}

\subsubsection{Type $B$}
	\label{sec:dyck_paths_b}
By definition, a Dyck path of type $B$ having semilength $n$ is a path that stays weakly above the line $x=y$, and consists of $2n$ steps.  Observe that the definition of the reflection map $\psi$ in the last section does not require its input to be a Dyck path of type $A$, in fact, we can apply it to any word over the alphabet $\{u,r\}$.  It is easy to see that in general for $\pf\in D_{n}^{B}$, the image $\psi(\pf)$ need not be a Dyck path again.  The next lemma claims, however, that the concatenation of $\pf\in D_{n}^{B}$ and $\psi(\pf)$ is in $D_{2n}^{A}$.  

\begin{lemma}\label{lem:involution_concatenation_b}
	For any $\pf\in D_{n}^{B}$ the concatenation of $\pf$ and $\psi(\pf)$ is in $D_{2n}^{A}$.
\end{lemma}
\begin{proof}
	Observe that if $w_{\pf}$ has $k$ letters equal to $r$, then $\pf$ ends at $(k,2n-k)$.  By definition $\pf$ does not cross the diagonal $x=y$, and therefore the concatenation of $\pf$ and $\psi(\pf)$ does not cross this diagonal either.  It thus corresponds to a path in $D_{2n}^{A}$. 
\end{proof}

\begin{corollary}\label{cor:dyck_b_equalizer}
	The sets $D_{n}^{B}$ and $\bigl\{\pf\in D_{2n}^{A}\mid\psi(\pf)=\pf\bigr\}$ are in bijection. 
\end{corollary}
\begin{proof}
	Lemma~\ref{lem:involution_concatenation_b} describes the map that sends each $\pf\in D_{n}^{B}$ to some $\qf\in D_{2n}^{A}$, and it is immediate from the construction that in this case $\psi(\qf)=\qf$.  Conversely, if we have $\qf\in D_{2n}^{A}$ with $\psi(\qf)=\qf$, then this path is completely determined by its first $2n$ steps, which can be regarded as a path $\pf\in D_{n}^{B}$ in its own right.  That these two maps are mutually inverse is straightforward to verify.
\end{proof}

We call the Dyck paths in $\bigl\{\pf\in D_{2n}^{A}\mid\psi(\pf)=\pf\bigr\}$ \alert{centrally symmetric}.  Now we can use the connection described in Corollary~\ref{cor:dyck_b_equalizer} to define a height sequence for Dyck paths of type $B$.  Let $\pf\in D_{2n}^{A}$ have $\psi(\pf)=\pf$, and let $\qf\in D_{n}^{B}$ denote the subpath consisting of the first $2n$ steps.  Let $\hh_{\pf}=(h_{1},h_{2},\ldots,h_{2n})$, and define $k=\min\bigl\{i\in[n]\mid i+h_{i}\geq 2n\bigr\}$.  The sequence $\hh_{\qf}=(h_{1},h_{2},\ldots,h_{k})$ is the \alert{height sequence} of $\qf$.  We observe that $k+h_{k}=2n$ if $w_{\qf}$ ends with a right-step, and $k+h_{k}=2n+1$ if $w_{\qf}$ ends with an up-step.  It is clear that $h_{i}$ is precisely the number of up-steps occuring before the $i\th$ occurrence of the letter $r$ in $w_{\qf}$, with the exception that if $w_{\qf}$ ends with the letter $u$, then $h_{k}$ is the total number of $u$'s occurring in $w_{\qf}$.  Like in the type $A$ case, the entry $h_{i}$ in $\hh_{\pf}$ determines precisely the height of $\pf$ at abscissa $i-1$. 

\begin{lemma}\label{lem:height_sequence_b}
	If $\pf\in D_{n}^{B}$, then $\hh_{\pf}=(h_{1},h_{2},\ldots,h_{k})$ for $k\in[n]$ satisfies
	\begin{displaymath}
		h_{k} = \begin{cases}2n-k+1, & \text{if}\;w_{\pf}\;\text{ends with}\;u\\ 2n-k, & \text{if}\;w_{\pf}\;\text{ends with}\;r,\end{cases}
	\end{displaymath}
	as well as $h_{1}\leq h_{2}\leq\cdots\leq h_{k-1}\leq 2n-k$ and $h_{i}\geq i$ for $i\in[k]$.  Conversely, each such sequence uniquely determines a Dyck path in $D_{n}^{B}$.  
\end{lemma}
\begin{proof}
	If $w_{\pf}$ is a Dyck word of type $B$, then it is straightforward to verify that $\hh_{\pf}$ satisfies the given conditions.
	
	\smallskip
	
	Conversely, let $\hh=(h_{1},h_{2},\ldots,h_{k})$ satisfy the given conditions, where we put $h_{0}=0$ in the case $k=1$.  If $h_{k}=2n-k+1$, then we define 
	\begin{displaymath}
		w_{\hh} = \underbrace{uu\cdots u}_{h_{1}}r\underbrace{uu\cdots u}_{h_{2}-h_{1}}r\cdots r\underbrace{uu\cdots u}_{h_{k}-h_{k-1}},
	\end{displaymath}
	and if $h_{k}=2n-k$, then we define 
	\begin{displaymath}
		w_{\hh} = \underbrace{uu\cdots u}_{h_{1}}r\underbrace{uu\cdots u}_{h_{2}-h_{1}}r\cdots r\underbrace{uu\cdots u}_{h_{k}-h_{k-1}}r.
	\end{displaymath}
	In both cases, it is straightforward to verify that $w_{\hh}$ is a Dyck word of type $B$.
\end{proof}

The next lemma describes how to derive the height sequence of the centrally symmetric Dyck path $\pf\in D_{2n}^{A}$ from a Dyck path $\qf\in D_{n}^{B}$.  

\begin{lemma}\label{lem:dyck_b_automorphism_height}
	Let $\qf\in D_{n}^{B}$ with height sequence $\hh_{\qf}=(h_{1},h_{2},\ldots,h_{k})$.  The corresponding centrally symmetric Dyck path $\pf\in D_{2n}^{A}$ has height sequence $\hh_{\pf}=(h'_{1},h'_{2},\ldots,h'_{2n})$, where $h'_{i}=h_{i}$ for $i\leq k$, and 
	\begin{displaymath}
		h'_{2n-h_{i}+1} = h'_{2n-h_{i}+2} = \cdots = h'_{2n-h_{i-1}} = 2n-i+1,
	\end{displaymath}
	for $i\in[k]$ with $h_{i}>h_{i-1}$ and $h_{0}=0$.  
\end{lemma}
\begin{proof}
	This is essentially the same proof as the one of Lemma~\ref{lem:involution_a_height}.  
\end{proof}

\begin{remark}\label{rem:dyck_b_reduction}
	If $\pf\in D_{n}^{A}\subseteq D_{n}^{B}$, then the associated Dyck word $w_{\pf}$ ends with the letter $r$, and its height sequence has precisely $n$ entries.  In this case, the conditions in Lemma~\ref{lem:height_sequence_b} coincide with those in Lemma~\ref{lem:height_sequence_a}.
\end{remark}

\begin{lemma}\label{lem:dyck_b_comparison}
	Let $\pf,\pf'\in D_{n}^{B}$ with associated height sequences $\hh_{\pf}=(h_{1},h_{2},\ldots,h_{k})$ and $\hh_{\pf'}=(h'_{1},h'_{2},\ldots,h'_{k'})$, respectively.  If $k<k'$, then $h_{k}>h'_{k}$.
\end{lemma}
\begin{proof}
	Lemma~\ref{lem:height_sequence_b} implies that $h_{k}\in\{2n-k,2n-k+1\}$, and $h_{k'}\in\{2n-k',2n-k'+1\}$.  If $k<k'$, then we immediately get $2n-k+1>2n-k>2n-k'$ and $2n-k+1>2n-k'+1$, as well as $2n-k\geq 2n-k'+1$.  If $2n-k=2n-k'+1$, then both paths $\pf$ and $\pf'$ end at the same height, which forces $k=k'$, a contradiction. 
\end{proof}

\begin{figure}
	\centering
	\begin{tikzpicture}\small
		\def\x{1.5};
		\def\y{1.45};
		\draw(2*\x,1*\y) node(v1){\dyckBThree{0/.4/.4/.4/.4/.8/.8/.8/.8/1.2/1.2/1.2}{\cTwo}{.4}};
		\draw(1*\x,2*\y) node(v2){\dyckBThree{0/.4/0/.8/.4/.8/.8/.8/.8/1.2/1.2/1.2}{\cTwo}{.4}};
		\draw(2*\x,2*\y) node(v3){\dyckBThree{0/.4/.4/.4/.4/.8/.4/1.2/.8/1.2/1.2/1.2}{\cTwo}{.4}};
		\draw(3*\x,2*\y) node(v4){\dyckBThree{0/.4/.4/.4/.4/.8/.8/.8/.8/1.2/.8/1.6}{\cTwo}{.4}};
		\draw(1*\x,3*\y) node(v5){\dyckBThree{0/.4/0/.8/.4/.8/.4/1.2/.8/1.2/1.2/1.2}{\cOne}{.4}};
		\draw(2*\x,3*\y) node(v6){\dyckBThree{0/.4/0/.8/.4/.8/.8/.8/.8/1.2/.8/1.6}{\cTwo}{.4}};
		\draw(3*\x,3*\y) node(v7){\dyckBThree{0/.4/.4/.4/.4/.8/.4/1.2/.8/1.2/.8/1.6}{\cOne}{.4}};
		\draw(1*\x,4*\y) node(v8){\dyckBThree{0/.4/0/.8/0/1.2/.4/1.2/.8/1.2/1.2/1.2}{\cTwo}{.4}};
		\draw(2*\x,4*\y) node(v9){\dyckBThree{0/.4/0/.8/.4/.8/.4/1.2/.8/1.2/.8/1.6}{\cOne}{.4}};
		\draw(3*\x,4*\y) node(v10){\dyckBThree{0/.4/.4/.4/.4/.8/.4/1.2/.4/1.6/.8/1.6}{\cOne}{.4}};
		\draw(1*\x,5*\y) node(v11){\dyckBThree{0/.4/0/.8/0/1.2/.4/1.2/.8/1.2/.8/1.6}{\cOne}{.4}};
		\draw(2*\x,5*\y) node(v12){\dyckBThree{0/.4/0/.8/.4/.8/.4/1.2/.4/1.6/.8/1.6}{\cOne}{.4}};
		\draw(3*\x,5*\y) node(v13){\dyckBThree{0/.4/.4/.4/.4/.8/.4/1.2/.4/1.6/.4/2}{\cTwo}{.4}};
		\draw(1.5*\x,6*\y) node(v14){\dyckBThree{0/.4/0/.8/0/1.2/.4/1.2/.4/1.6/.8/1.6}{\cOne}{.4}};
		\draw(2.5*\x,6*\y) node(v15){\dyckBThree{0/.4/0/.8/.4/.8/.4/1.2/.4/1.6/.4/2}{\cOne}{.4}};
		\draw(1.5*\x,7*\y) node(v16){\dyckBThree{0/.4/0/.8/0/1.2/0/1.6/.4/1.6/.8/1.6}{\cOne}{.4}};
		\draw(2.5*\x,7*\y) node(v17){\dyckBThree{0/.4/0/.8/0/1.2/.4/1.2/.4/1.6/.4/2}{\cOne}{.4}};
		\draw(2*\x,8*\y) node(v18){\dyckBThree{0/.4/0/.8/0/1.2/0/1.6/.4/1.6/.4/2}{\cOne}{.4}};
		\draw(2*\x,9*\y) node(v19){\dyckBThree{0/.4/0/.8/0/1.2/0/1.6/0/2/.4/2}{\cOne}{.4}};
		\draw(2*\x,10*\y) node(v20){\dyckBThree{0/.4/0/.8/0/1.2/0/1.6/0/2/0/2.4}{\cTwo}{.4}};
		\draw(v1) -- (v2);
		\draw(v1) -- (v3);
		\draw(v1) -- (v4);
		\draw(v2) -- (v5);
		\draw(v2) -- (v6);
		\draw(v3) -- (v5);
		\draw(v3) -- (v7);
		\draw(v4) -- (v6);
		\draw(v4) -- (v7);
		\draw(v5) -- (v8);
		\draw(v5) -- (v9);
		\draw(v6) -- (v9);
		\draw(v7) -- (v9);
		\draw(v7) -- (v10);
		\draw(v8) -- (v11);
		\draw(v9) -- (v11);
		\draw(v9) -- (v12);
		\draw(v10) -- (v12);
		\draw(v10) -- (v13);
		\draw(v11) -- (v14);
		\draw(v12) -- (v14);
		\draw(v12) -- (v15);
		\draw(v13) -- (v15);
		\draw(v14) -- (v16);
		\draw(v14) -- (v17);
		\draw(v15) -- (v17);
		\draw(v16) -- (v18);
		\draw(v17) -- (v18);
		\draw(v18) -- (v19);
		\draw(v19) -- (v20);
	\end{tikzpicture}
	\caption{The lattice $\DD_{3}^{B}$.  The highlighted paths are regular.}
	\label{fig:dominance_b3}
\end{figure}

Let $\pf,\pf'\in D_{n}^{B}$ have associated height sequences $\hh_{\pf}=(h_{1},h_{2},\ldots,h_{k})$ and $\hh_{\pf'}=(h'_{1},h'_{2},\ldots,h'_{k'})$ for $k,k'\in[n]$.  Define $\pf\leq_{D}\pf'$ if and only if $k\geq k'$ and $h_{i}\leq h'_{i}$ for $i\in[k']$, and call this partial order the \alert{dominance order on $D_{n}^{B}$}.  We usually write $\DD_{n}^{B}$ for the poset $\bigl(D_{n}^{B},\leq_{D}\bigr)$.  Figure~\ref{fig:dominance_b3} shows $\DD_{3}^{B}$.  The following result extends Theorem~\ref{thm:distributive_a}.

\begin{theorem}\label{thm:distributive_b}
	For $n>0$ the poset $\DD_{n}^{B}$ is a distributive lattice.
\end{theorem}
\begin{proof}
	Let $\pf,\pf'\in D_{n}^{B}$ have height sequences $\hh_{\pf}=(h_{1},h_{2},\ldots,h_{k})$ and $\hh_{\pf'}=(h'_{1},h'_{2},\ldots,h'_{k'})$, and assume without loss of generality that $k\geq k'$.  It is straightforward to verify that their meet can be defined via the height sequence
	\begin{displaymath}
		\hh_{\pf\wedge_{D}\pf'} = \bigl(\min\{h_{1},h'_{1}\},\min\{h_{2},h'_{2}\},\ldots,\min\{h_{k'},h'_{k'}\},h_{k'+1},h_{k'+2},\ldots,h_{k}\bigr),
	\end{displaymath}
	and their join can be defined via the height sequence
	\begin{displaymath}
		\hh_{\pf\vee_{D}\pf'} = \bigl(\max\{h_{1},h'_{1}\},\max\{h_{2},h'_{2}\},\ldots,\max\{h_{k'},h'_{k'}\}\bigr).
	\end{displaymath}
	Since $\min$ and $\max$ are distributive, the result follows.
\end{proof}

In fact, we can strengthen Corollary~\ref{cor:dyck_b_equalizer} as follows.

\begin{lemma}\label{lem:dyck_b_equalizer_poset}
	The reflection map $\psi$ is a poset isomorphism from $\DD_{n}^{B}$ to $\Bigl(\bigl\{\qf\in D_{2n}^{A}\mid\psi(\qf)=\qf\bigr\},\leq_{D}\Bigr)$.  
\end{lemma}
\begin{proof}
	Let $\pf,\pf'\in D_{n}^{B}$.  If $\pf\leq_{D}\pf'$, then $\pf$ stays weakly below $\pf'$, and by construction the same is true for $\psi(\pf)$ and $\psi(\pf')$.  Let $\qf$ and $\qf'$ denote the concatenation of $\pf$ and $\psi(\pf)$, respectively $\pf'$ and $\psi(\pf')$.  We clearly have $\psi(\qf)=\qf$ and $\psi(\qf')=\qf'$, and the previous shows that $\qf$ stays weakly below $\qf'$, which yields $\qf\leq_{D}\qf'$.  The opposite direction is trivial.
\end{proof}

\subsection{Join-Prime Dyck Paths}
	\label{sec:join_prime_paths}
In general, many properties of a Heyting algebra can be understood by looking at the induced subposet of join-prime elements.  An element $p$ in a finite lattice $\LL=(L,\leq)$ is \alert{join-prime} if $p$ is not the least element, and if for any two elements $x,y\in L$ we have that $p\leq x\vee y$ implies $p\leq x$ or $p\leq y$.  Moreover, $p$ is \alert{join-irreducible} if it is not the least element, and for any $X\subseteq L$ with $p=\bigvee X$ it follows that $p\in X$.  In other words, join-irreducible elements in a finite lattice are precisely those elements with a unique lower cover.  We have the following straightforward result.

\begin{lemma}[\cite{balbes74distributive}*{Theorem~III.1.2(i)}]\label{lem:distributive_prime_irreducible}
	If $\LL$ is a finite distributive lattice, then an element is join-prime if and only if it is join-irreducible.
\end{lemma}
%\begin{proof}
%	Let $\LL=(L,\leq)$ be distributive.  If $p\in L$ is not join-irreducible, then there exist two distinct lower covers $x,y\in L$ of $p$, and we have $p=x\vee y$, but $p\not\leq x$ and $p\not\leq y$.  Hence $p$ is not join-prime.  
%	
%	On the other hand, if $p\in L$ is join-irreducible, and $x,y\in L$ satisfy $p\leq x\vee y$, then we have $p=p\wedge(x\vee y)=(p\wedge x)\vee (p\wedge y)$.  It follows that $p=p\wedge x$ or $p=p\wedge y$, which implies $p\leq x$ or $p\leq y$, respectively.  Hence $p$ is join-prime.
%\end{proof}

Let us now characterize the join-prime elements of $\DD_{n}^{B}$. 

\begin{lemma}\label{lem:join_prime_b}
	An element $\pf\in D_{n}^{B}$ with height sequence $\hh_{\pf}=(h_{1},h_{2},\ldots,h_{k})$ is join-prime if and only if there exists a unique index $i\in[k]$ such that $h_{i}>i$ and $h_{i}>h_{i-1}$, where we put $h_{0}=0$ if necessary. 
\end{lemma}
\begin{proof}
	In view of Theorem~\ref{thm:distributive_b} and Lemma~\ref{lem:distributive_prime_irreducible} it suffices to characterize the join-irreducible elements of $\DD_{n}^{B}$.
	
	First suppose that $\hh_{\pf}$ has the desired form, \ie there exists a unique index $i\in[k]$ such that $h_{i}>i$ and $h_{i}>h_{i-1}$.  We necessarily have $h_{i-1}=i-1$, and $h_{i}=h_{i+1}=\cdots=h_{\min\{h_{i},k\}}$.  It is now easy to check that $(h_{1},h_{2},\ldots,h_{i-1},h_{i}-1,h_{i+1},\ldots,h_{k})$ determines a Dyck path $\pf'\in D_{n}^{B}$, which is the unique lower cover of $\pf$ in $\DD_{n}^{B}$.  Hence $\pf$ is join-irreducible.
		
	Conversely, suppose that $\hh_{\pf}$ is not of the desired form.  Then we may have that $h_{i}=i$ for all $i\in[k]$, which implies that $\pf$ is the least element of $\DD_{n}^{B}$, and is thus not join-irreducible.  Otherwise, there are at least two indices $i$ and $j$ with $j>i$ such that $h_{i}>i$ and $h_{i}>h_{i-1}$ as well as $h_{j}>j$ and $h_{j}>h_{j-1}$.  It follows that both $(h_{1},h_{2},\ldots,h_{i-1},h_{i}-1,h_{i+1},\ldots,h_{k})$ and $(h_{1},h_{2},\ldots,h_{j-1},h_{j}-1,h_{j+1},\ldots,h_{k})$ determine lower covers of $\pf$ that are mutually incomparable.  Hence $\pf$ is not join-irreducible.
\end{proof}

The join-prime elements of $\DD_{n}^{A}$ are recovered by considering the case $k=n$ in Lemma~\ref{lem:join_prime_b}, and they have been characterized previously in \cite{ferrari11lattices}*{Proposition~3.7}.

As a final part of this section, we describe the subposet of join-prime elements of $\DD_{n}^{A}$ and $\DD_{n}^{B}$, respectively.  First of all, let us denote the set of join-irreducible elements of a finite lattice $\LL$ by $J(\LL)$, and write $J_{n}^{A}=J\bigl(\DD_{n}^{A}\bigr)$ and $J_{n}^{B}=J\bigl(\DD_{n}^{B}\bigr)$.

Define $T_{n}^{A}=\bigl\{(i,j)\mid 1\leq i<j\leq n\bigr\}$ and $T_{n}^{B}=\bigl\{(i,j)\mid 1\leq i<j\leq 2n+1-i\bigr\}$.  Consider the partial order on $\mathbb{N}^{2}$ defined by $(a,b)\leq (a',b')$ if and only if $a\geq a'$ and $b\leq b'$.  

\begin{lemma}\label{lem:join_prime_posets}
	We have $\bigl(J_{n}^{A},\leq_{D}\bigr)\cong\bigl(T_{n}^{A},\leq)$ and $\bigl(J_{n}^{B},\leq_{D}\bigr)\cong\bigl(T_{n}^{B},\leq\bigr)$ for every $n>0$.
\end{lemma}
\begin{proof}
	It is sufficient to consider only the type-$B$ case, since type $A$ is just a restriction of it.  
	
	Lemma~\ref{lem:join_prime_b} implies that the elements in $J_{n}^{B}$ can be indexed by a pair $(i,h_{i})$ for $i\in[k]$, where $k$ is the length of the corresponding height sequence, and $h_{i}$ is some value in $\{i+1,i+2,\ldots,2n+1-i\}$.  Write $\pf_{ij}$ for the corresponding path, where $j=h_{i}$.  Now pick $\pf_{ij},\pf_{i'j'}\in J_{n}^{B}$ with height sequences $\hh_{\pf_{ij}}=(h_{1},h_{2},\ldots,h_{k})$ and $\hh_{\pf_{i'j'}}=(h'_{1},h'_{2},\ldots,h'_{k'})$.  Lemma~\ref{lem:join_prime_b} implies that $h_{i}=j$ is the first entry of $\hh_{\pf}$ that is strictly greater than its index, and likewise for $h'_{i'}=j'$ and $\hh_{\pf_{i'j'}}$.  
	
	Suppose that $\pf_{ij}\leq_{D}\pf_{i'j'}$.  If $i<i'$, then $h'_{i}=i$, and we get $j=h_{i}\leq h'_{i}=i<j$, which is a contradiction.  If $i\geq i'$ and $j>j'$, then $j=h_{i}\leq h'_{i}=\max\{i,j'\}<j$, which is a contradiction.  Hence $(i,j)\leq(i',j')$ as desired.
	
	Conversely suppose that $i'\leq i<j\leq j'$, and let $s\in[n]$.  If $s<i'$, then $h_{s}=s=h'_{s}$.  If $i'\leq s<i$, then $h'_{s}=j'$, and $h_{s}=s$, which implies $h_{s}\leq h'_{s}$.  If $i\leq s$, then $h_{s}=\max\{j,s\}\leq\max\{j',s\}=h'_{s}$.  Hence $\pf_{ij}\leq_{D}\pf_{i'j'}$.  
\end{proof}

We remark that the type $A$-case of Lemma~\ref{lem:join_prime_posets} follows also from \cite{ferrari11lattices}*{Theorem~3.4}.  

\begin{corollary}\label{cor:dyck_paths_root_ideals}
	For $n>0$ we have $\DD_{n}^{A}\cong\II\bigl((T_{n}^{A},\leq)\bigr)$ and $\DD_{n}^{B}\cong\II\bigl((T_{n}^{B},\leq\bigr)$.  
\end{corollary}
\begin{proof}
	This follows from Theorems~\ref{thm:distributive_a} and \ref{thm:distributive_b}, and Lemma~\ref{lem:join_prime_posets} using G.~Birkhoff's representation theorem of finite distributive lattices~\cite{birkhoff37rings}*{Theorem~5}.
\end{proof}

\section{Formulas in the Heyting Algebra of Dyck Paths}
	\label{sec:formulas}
We begin this section with the proof of Theorem~\ref{thm:dyck_algebras}, since we have already gathered all the ingredients.  Subsequently, we provide explicit, combinatorial formulas for the computation of relative pseudocomplents and pseudocomplements in the Heyting algebras of Dyck paths of type $A$ and $B$, respectively, and we identify the regular elements.  

\begin{proof}[Proof of Theorem~\ref{thm:dyck_algebras}]
	The fact that both $\DD_{n}^{A}$ and $\DD_{n}^{B}$ are Heyting algebras follows from Theorem~\ref{thm:distributive_heyting} as well as Theorems~\ref{thm:distributive_a} and \ref{thm:distributive_b}, respectively.  The fact that $\DD_{n}^{A}$ is not a Heyting subalgebra of $\DD_{n}^{B}$ follows easily from Lemma~\ref{lem:pseudocomplement_comparable}.  Pick $\pf\in D_{n}^{A}$.  Said lemma implies $\pf\to^{A}\pf=\mathfrak{1}^{A}$ and $\pf\to^{B}\pf=\mathfrak{1}^{B}$, but $\mathfrak{1}^{A}<_{D}\mathfrak{1}^{B}$ in $\DD_{n}^{B}$.  The fact that $\DD_{n}^{A}$ is a sublattice of $\DD_{n}^{B}$ is a straightforward computation.  
	
	For the last part, recall that $\{\pf\in D_{2n}^{A}\mid\psi(\pf)=\pf\bigr\}$ is precisely the equalizer $\text{Eq}(\text{id},\psi)$.  Now, Lemma~\ref{lem:dyck_b_equalizer_poset} states that $\DD_{n}^{B}$ is isomorphic (as a lattice) to $\text{Eq}(\text{id},\psi)$.  Proposition~\ref{prop:equalizer_subobject} implies that $\text{Eq}(\text{id},\psi)$ is a Heyting subalgebra of $\DD_{2n}^{A}$, and since the operation $\to^{A}$ is defined completely in lattice-theoretic terms, the claim follows.
\end{proof}

\subsection{Monotone Lattice Paths}
	\label{sec:lattice_paths}
Before we compute the explicit formulas for the height functions of (relative) pseudocomplements and regular elements in the Heyting algebras of Dyck paths of type $A$ and $B$, we consider a slightly more general setting.  More precisely, for the moment we consider all lattice paths from $(0,0)$ to $(n,m)$ consisting only of up- and right-steps, and denote the set of all such paths by $L_{n,m}$.  Any $\pf\in L_{n,m}$ is uniquely determined by a \alert{height sequence} $\hh_{\pf}=(h_{1},h_{2},\ldots,h_{n})$, where $h_{1}\leq h_{2}\leq\cdots\leq h_{n}\leq m$.  Again, we consider this set equipped with the \alert{dominance order}, defined by $\pf\leq_{D}\pf'$ whenever $\hh_{\pf}$ is componentwise weakly smaller than $\hh_{\pf'}$.  It is straightforward to verify that the resulting poset $\LL_{n,m}=\bigl(L_{n,m},\leq_{D}\bigr)$ is a distributive lattice, and in view of Theorem~\ref{thm:distributive_heyting} it also forms a Heyting algebra.  See Figure~\ref{fig:paths_3} for an illustration.  

\begin{figure}
	\centering
	\begin{tikzpicture}\small
		\def\x{1.5};
		\def\y{1.15};
		\draw(2*\x,1*\y) node(v1){\pathThree{0/0/0}{\cTwo}{.4}{}};
		\draw(2*\x,2*\y) node(v2){\pathThree{0/0/1}{\cOne}{.4}{}};
		\draw(1.5*\x,3*\y) node(v3){\pathThree{0/1/1}{\cOne}{.4}{}};
		\draw(2.5*\x,3*\y) node(v4){\pathThree{0/0/2}{\cOne}{.4}{}};
		\draw(1*\x,4*\y) node(v5){\pathThree{1/1/1}{\cOne}{.4}{}};
		\draw(2*\x,4*\y) node(v6){\pathThree{0/1/2}{\cOne}{.4}{}};
		\draw(3*\x,4*\y) node(v7){\pathThree{0/0/3}{\cOne}{.4}{}};
		\draw(1*\x,5*\y) node(v8){\pathThree{1/1/2}{\cOne}{.4}{}};
		\draw(2*\x,5*\y) node(v9){\pathThree{0/2/2}{\cOne}{.4}{}};
		\draw(3*\x,5*\y) node(v10){\pathThree{0/1/3}{\cOne}{.4}{}};
		\draw(1*\x,6*\y) node(v11){\pathThree{1/2/2}{\cOne}{.4}{}};
		\draw(2*\x,6*\y) node(v12){\pathThree{1/1/3}{\cOne}{.4}{}};
		\draw(3*\x,6*\y) node(v13){\pathThree{0/2/3}{\cOne}{.4}{}};
		\draw(1*\x,7*\y) node(v14){\pathThree{2/2/2}{\cOne}{.4}{}};
		\draw(2*\x,7*\y) node(v15){\pathThree{1/2/3}{\cOne}{.4}{}};
		\draw(3*\x,7*\y) node(v16){\pathThree{0/3/3}{\cOne}{.4}{}};
		\draw(1.5*\x,8*\y) node(v17){\pathThree{2/2/3}{\cOne}{.4}{}};
		\draw(2.5*\x,8*\y) node(v18){\pathThree{1/3/3}{\cOne}{.4}{}};
		\draw(2*\x,9*\y) node(v19){\pathThree{2/3/3}{\cOne}{.4}{}};
		\draw(2*\x,10*\y) node(v20){\pathThree{3/3/3}{\cTwo}{.4}{}};
		\draw(v1) -- (v2);
		\draw(v2) -- (v3);
		\draw(v2) -- (v4);
		\draw(v3) -- (v5);
		\draw(v3) -- (v6);
		\draw(v4) -- (v6);
		\draw(v4) -- (v7);
		\draw(v5) -- (v8);
		\draw(v6) -- (v8);
		\draw(v6) -- (v9);
		\draw(v6) -- (v10);
		\draw(v7) -- (v10);
		\draw(v8) -- (v11);
		\draw(v8) -- (v12);
		\draw(v9) -- (v11);
		\draw(v9) -- (v13);
		\draw(v10) -- (v12);
		\draw(v10) -- (v13);
		\draw(v11) -- (v14);
		\draw(v11) -- (v15);
		\draw(v12) -- (v15);
		\draw(v13) -- (v15);
		\draw(v13) -- (v16);
		\draw(v14) -- (v17);
		\draw(v15) -- (v17);
		\draw(v15) -- (v18);
		\draw(v16) -- (v18);
		\draw(v17) -- (v19);
		\draw(v18) -- (v19);
		\draw(v19) -- (v20);		
	\end{tikzpicture}		
	\caption{The lattice $\LL_{3,3}$.  The highlighted paths are regular.}
	\label{fig:paths_3}
\end{figure}

It is straightforward to verify that $\DD_{n}^{A}$ is isomorphic to the interval $[\mathfrak{0}^{A},\mathfrak{1}]$ in $\LL_{n,n}$, where $\hh_{\mathfrak{0}^{A}}=(1,2,\ldots,n)$ and $\hh_{\mathfrak{1}}=(n,n,\ldots,n)$.  Moreover, $\DD_{n+1}^{A}$ is isomorphic to the interval $[\qf,\mathfrak{1}]$ in $\LL_{n,n}$, where $\hh_{\qf}=(0,1,\ldots,n)$.  This is because if $\pf\in D_{n+1}^{A}$ has height sequence $\hh_{\pf}=(h_{1},h_{2},\ldots,h_{n+1})$, then $h_{i}\geq i$, and $h_{n+1}=n+1$.  So the sequence $(h_{1}-1,h_{2}-1,\ldots,h_{n}-1)$ is the height sequence of some path in $[\qf,\mathfrak{1}]$, and this is clearly a bijection.  The following theorem states an explicit formula for the computation of the relative pseudocomplement in $\LL_{n,m}$ in terms of height sequences. 

\begin{theorem}\label{thm:rel_pseudocomp_paths}
	Let $\pf_{1},\pf_{2}\in L_{n,m}$ with height sequences $\hh_{\pf_{1}}=\bigl(h_{1}^{(1)},h_{2}^{(1)},\ldots,h_{n}^{(1)}\bigr)$ and $\hh_{\pf_{2}}=\bigl(h_{1}^{(2)},h_{2}^{(2)},\ldots,h_{n}^{(2)}\bigr)$.  The relative pseudocomplement $\pf_{1}\to_{D}\pf_{2}$ is the path $\pf\in L_{n,m}$ determined by the height sequence $\hh_{\pf}=(h_{1},h_{2},\ldots,h_{n})$, with
	\begin{displaymath}
		h_{i} = \begin{cases}m & \text{if}\; i=n\;\text{and}\;h_{i}^{(1)}\leq h_{i}^{(2)},\\h_{i+1}, & \text{if}\;i<n\;\text{and}\;h_{i}^{(1)}\leq h_{i}^{(2)},\\h_{i}^{(2)}, & \text{if}\;i\leq n\;\text{and}\;h_{i}^{(1)}>h_{i}^{(2)}.\end{cases}
	\end{displaymath}
\end{theorem}
\begin{proof}
	First of all we need to show that $\hh_{\pf}$ is indeed a height sequence of some $\pf\in L_{n,m}$, which means that we need to show that $h_{i}\leq h_{i+1}$ for all $i\in[n-1]$ and $h_{n}\leq m$.  The second condition is immediate from the construction.  For the first condition, we observe that there are two options.  Either $h_{i}=h_{i+1}$, and we are done, or $h_{i}=h_{i}^{(2)}$.  If $h_{i}=h_{i}^{(2)}$, then we need to look at $h_{i+1}$.  Either we have $h_{i+1}=h_{i+2}=\cdots=h_{n}$, where $h_{n}\in\{m,h_{n}^{(2)}\}$, and we are done since $h_{i}^{(2)}\leq h_{n}^{(2)}\leq m$ by virtue of the fact that $\pf_{2}\in L_{n,m}$, or there is some minimal $j>i$ with $h_{j}=h_{j}^{(2)}$.  In that case we have $h_{i+1}=h_{i+2}=\cdots=h_{j}=h_{j}^{(2)}$, and again we are done since $h_{i}^{(2)}\leq h_{j}^{(2)}$.  
	
	Now we need to show that $\pf$ satisfies $\pf_{1}\wedge_{D}\pf\leq_{D}\pf_{2}$.  Since the meet in $\LL_{n,m}$ is given by taking the componentwise minimum of the height sequences, it suffices to show that $\min\bigl\{h_{i}^{(1)},h_{i}\bigr\}\leq h_{i}^{(2)}$ for all $i\in[n]$.  If $h_{i}^{(1)}>h_{i}^{(2)}$, then we have $\min\bigl\{h_{i}^{(1)},h_{i}\bigr\}=\min\bigl\{h_{i}^{(1)},h_{i}^{(2)}\bigr\}\leq h_{i}^{(2)}$; and if $h_{i}^{(1)}\leq h_{i}^{(2)}$, then we can check that $h_{i}^{(1)}\leq h_{i}$, which yields $\min\bigl\{h_{i}^{(1)},h_{i}\bigr\}=h_{i}^{(1)}\leq h_{i}^{(2)}$ as desired.
	
%	If $h_{i}^{(1)}>h_{i}^{(2)}$, then we have $\min\bigl\{h_{i}^{(1)},h_{i}\bigr\}=\min\bigl\{h_{i}^{(1)},h_{i}^{(2)}\bigr\}\leq h_{i}^{(2)}$ as desired.  So suppose $h_{i}^{(1)}\leq h_{i}^{(2)}$.  If $i=n$, then $h_{i}=m$, which implies $\min\bigl\{h_{i}^{(1)},h_{i}\bigr\}=h_{i}^{(1)}\leq h_{i}^{(2)}$.  If $i<n$, then there are two possibilities.  Either $h_{i}=h_{i+1}=\cdots=h_{n}=m$, or there is some minimal $j\in\{i,i+1,\ldots,n\}$ such that $h_{i}=h_{i+1}=\cdots=h_{j}=h_{j}^{(2)}$.  In the former case, we again get $\min\bigl\{h_{i}^{(1)},h_{i}\bigr\}=h_{i}^{(1)}\leq h_{i}^{(2)}$.  The latter case can only be achieved if $h_{j}^{(2)}<h_{j}^{(1)}$, and the minimality of $j$ implies $h_{i}^{(1)}\leq h_{j}^{(2)}=h_{i}$, which yields $\min\bigl\{h_{i}^{(1)},h_{i}\bigr\}=h_{i}^{(1)}\leq h_{i}^{(2)}$.
	
	Now suppose that there is some other $\pf'\in L_{n,m}$ with $\pf_{1}\wedge_{D}\pf'\leq\pf_{2}$, and let $\hh_{\pf'}=(h'_{1},h'_{2},\ldots,h'_{n})$.  If $\pf<_{D}\pf'$, then there must be a maximal index $i\in[n]$ with $h_{i}<h'_{i}$.  If $i=n$, then $h_{n}<h'_{n}\leq m$, which rules out the option that $h_{n}=m$.  By construction, we therefore have $h_{n}=h_{n}^{(2)}$.  This can only occur if $h_{n}^{(1)}>h_{n}^{(2)}$.  The choice of $\pf'$ requires $\min\bigl\{h_{n}^{(1)},h'_{n}\bigr\}\leq h_{n}^{(2)}$, which yields the contradiction $h'_{n}\leq h_{n}^{(2)}=h_{n}<h'_{n}$.  Therefore, we have $i<n$.  If $h_{i}^{(1)}\leq h_{i}^{(2)}$, then $h_{i}=h_{i+1}$, and the maximality of $i$ implies $h'_{i}>h_{i+1}=h'_{i+1}$, which contradicts the assumption that $\pf'\in L_{n,m}$.  If $h_{i}^{(1)}>h_{i}^{(2)}$, then $h'_{i}>h_{i}=h_{i}^{(2)}$, and hence $\min\bigl\{h_{i}^{(1)},h'_{i}\bigr\}>h_{i}^{(2)}$, which contradicts $\pf_{1}\wedge_{D}\pf'\leq_{D}\pf_{2}$.  Hence $\pf$ is indeed the relative pseudocomplement of $\pf_{1}$ with respect to $\pf_{2}$.  
\end{proof}

The formula in Theorem~\ref{thm:rel_pseudocomp_paths} is inductive from the right.  We can rephrase Theorem~\ref{thm:rel_pseudocomp_paths} as follows.

\begin{corollary}\label{cor:rel_pseudocomp_paths}
	Let $\pf_{1},\pf_{2}\in L_{n,m}$ with height sequences $\hh_{\pf_{1}}=\bigl(h_{1}^{(1)},h_{2}^{(1)},\ldots,h_{n}^{(1)}\bigr)$ and $\hh_{\pf_{2}}=\bigl(h_{1}^{(2)},h_{2}^{(2)},\ldots,h_{n}^{(2)}\bigr)$, and set $\bigl\{i\in[n]\mid h_{i}^{(1)}>h_{i}^{(2)}\bigr\}=\{i_{1},i_{2},\ldots,i_{s}\}$ as well as $i_{0}=0$.  The relative pseudocomplement $\pf_{1}\to_{D}\pf_{2}$ is the path $\pf\in L_{n,m}$ determined by the height sequence $\hh_{\pf}=(h_{1},h_{2},\ldots,h_{n})$, with
	\begin{displaymath}
		h_{i_{k-1}+1}=h_{i_{k-1}+2}=\cdots=h_{i_{k}}=h_{i_{k}}^{(2)},
	\end{displaymath}
	for $k\in[s]$.  If $i_{s}<n$, then we additionally have $h_{i_{s}+1}=h_{i_{s}+2}=\cdots=h_{i_{n}}=m$.
\end{corollary}
\begin{proof}
	Pick $i\in[n]$.  If $h_{i}^{(1)}>h_{i}^{(2)}$, then $h_{i}=h_{i}^{(2)}$.  If $h_{i}^{(1)}\leq h_{i}^{(2)}$, then $h_{i}=h_{i+1}$ whenever $i<n$, and $h_{i}=m$, whenever $i>i_{s}$.  These are precisely the conditions in Theorem~\ref{thm:rel_pseudocomp_paths} for the height sequence of the relative pseudocomplement of $\pf_{1}$ with respect to $\pf_{2}$. 
\end{proof}

The following are immediate corollaries of Theorem~\ref{thm:rel_pseudocomp_paths}.

\begin{corollary}\label{cor:pseudocomp_paths}
	Let $\pf\in L_{n,m}$ with height sequence $\hh_{\pf}=(h_{1},h_{2},\ldots,h_{n})$.  The pseudocomplement of $\pf$ is either the least element $\mathfrak{0}$ whenever $\pf\neq\mathfrak{0}$, or the greatest element $\mathfrak{1}$ otherwise.
\end{corollary}
\begin{proof}
	We observe that the least element $\mathfrak{0}$ in $\LL_{n,m}$, which is given by the height sequence $(0,0,\ldots,0)$, has a unique upper cover, namely the path $\mathfrak{a}$ given by the height sequence $\hh_{\mathfrak{a}}=(0,0,\ldots,0,1)$.  Hence the greatest path $\qf$ such that $\pf\wedge_{D}\qf\leq_{D}\mathfrak{0}$ is $\qf=\mathfrak{0}$, unless $\pf=\mathfrak{0}$, in which case we get $\qf=\mathfrak{1}$.
\end{proof}

\begin{corollary}\label{cor:regular_paths}
	A path $\pf\in L_{n,m}$ is regular if and only if $\pf$ is either the least or the greatest element of $\LL_{n,m}$. 
\end{corollary}
\begin{proof}
	By definition $\pf\in L_{n,m}$ is regular if and only if $(\pf\to\mathfrak{0})\to\mathfrak{0}=\pf$.  If $\pf=\mathfrak{0}$, then Corollary~\ref{cor:pseudocomp_paths} implies $(\mathfrak{0}\to\mathfrak{0})\to\mathfrak{0}=\mathfrak{1}\to\mathfrak{0}=\mathfrak{0}$, and if $\pf\neq\mathfrak{0}$, then Corollary~\ref{cor:pseudocomp_paths} implies $(\pf\to\mathfrak{0})\to\mathfrak{0}=\mathfrak{0}\to\mathfrak{0}=\mathfrak{1}$, and hence $\pf=\mathfrak{1}$ is the only other regular element of $\LL_{n,m}$.  
\end{proof}

\subsection{Type $A$}
	\label{sec:formula_a}
As mentioned earlier, $\DD_{n}^{A}$ is an interval in $\LL_{n,n}$.  Hence we can use Theorem~\ref{thm:rel_pseudocomp_paths} to derive an explicit formula for the relative pseudocomplements in $\DD_{n}^{A}$. 

\begin{theorem}\label{thm:rel_pseudocomp_a}
		Let $\pf_{1},\pf_{2}\in D_{n,n}^{A}$ with height sequences $\hh_{\pf_{1}}=\bigl(h_{1}^{(1)},h_{2}^{(1)},\ldots,h_{n}^{(1)}\bigr)$ and $\hh_{\pf_{2}}=\bigl(h_{1}^{(2)},h_{2}^{(2)},\ldots,h_{n}^{(2)}\bigr)$, respectively.  The relative pseudocomplement $\pf_{1}\to_{D}\pf_{2}$ is the path $\pf\in D_{n}^{A}$ determined by the height sequence $\hh_{\pf}=(h_{1},h_{2},\ldots,h_{n})$, with
	\begin{displaymath}
		h_{i} = \begin{cases}h_{i+1}, & \text{if}\;i<n\;\text{and}\;h_{i}^{(1)}\leq h_{i}^{(2)},\\h_{i}^{(2)}, & \text{if}\;i=n,\;\text{or}\;i<n\;\text{and}\;h_{i}^{(1)}>h_{i}^{(2)}.\end{cases}
	\end{displaymath}
\end{theorem}
\begin{proof}
	Recall that $h_{n}^{(1)}=h_{n}^{(2)}=n$, hence the simplification of the statement (compared to Theorem~\ref{thm:rel_pseudocomp_paths}) is justified.   We only need to show that $\hh_{\pf}$ is indeed the height sequence of some $\pf\in D_{n}^{A}$.  Proving that $h_{1}\leq h_{2}\leq\cdots\leq h_{n}=n$ works analogously to the proof of Theorem~\ref{thm:rel_pseudocomp_paths}.  It remains to show that $h_{i}\geq i$ for $i\in[n]$.  This is clearly satisfied for $i=n$.  So choose $i<n$ maximal such that $h_{i}<i$.  We have two choices. Either $h_{i}=h_{i+1}$ and the maximality of $i$ implies that $i>h_{i}=h_{i+1}\geq i+1$ which is a contradiction, or $h_{i}=h_{i}^{(2)}$ and since $\hh_{\pf_{2}}$ is the height sequence of some $\pf\in D_{n}^{A}$ we conclude $i>h_{i}=h_{i}^{(2)}\geq i$ which is a contradiction as well.  We conclude that $h_{i}\geq i$ for all $i\in[n]$, and according to Lemma~\ref{lem:height_sequence_a} $\hh_{\pf}$ is indeed the height sequence of some $\pf\in\DD_{n}^{A}$.  The fact that $\pf$ is the relative pseudocomplement of $\pf_{1}$ with respect to $\pf_{2}$ works almost verbatim to the proof of Theorem~\ref{thm:rel_pseudocomp_paths}.  
\end{proof}

\begin{corollary}\label{cor:rel_pseudocomp_a}
	Let $\pf_{1},\pf_{2}\in D_{n}^{A}$ with height sequences $\hh_{\pf_{1}}=\bigl(h_{1}^{(1)},h_{2}^{(1)},\ldots,h_{n}^{(1)}\bigr)$ and $\hh_{\pf_{2}}=\bigl(h_{1}^{(2)},h_{2}^{(2)},\ldots,h_{n}^{(2)}\bigr)$, and set $\bigl\{i\in[n]\mid h_{i}^{(1)}>h_{i}^{(2)}\bigr\}=\{i_{1},i_{2},\ldots,i_{s}\}$ as well as $i_{0}=0$ and $i_{s+1}=n$.  The relative pseudocomplement $\pf_{1}\to_{D}\pf_{2}$ is the path $\pf\in D_{n}^{A}$ determined by the height sequence $\hh_{\pf}=(h_{1},h_{2},\ldots,h_{n})$, with
	\begin{displaymath}
		h_{i_{k-1}+1}=h_{i_{k}+2}=\cdots=h_{i_{k}}=h_{i_{k}}^{(2)},
	\end{displaymath}
	for $k\in[s+1]$.
\end{corollary}
\begin{proof}
	This is the statement of Corollary~\ref{cor:rel_pseudocomp_paths} adapted to the current situation.  Observe that we always have $i_{s}<n$ and $h_{n}^{(2)}=n$.
\end{proof}

\begin{corollary}\label{cor:pseudocomp_a}
	Let $\pf\in D_{n}^{A}$ with height sequence $\hh_{\pf}=(h_{1},h_{2},\ldots,h_{n})$.  The pseudocomplement of $\pf$ is the Dyck path $\pf^{\cc}\in D_{n}^{A}$ determined by the height sequence $\hh_{\pf^{c}}=(h_{1}^{\cc},h_{2}^{\cc},\ldots,h_{n}^{\cc})$ with
	\begin{displaymath}
		h_{i}^{\cc} = \begin{cases}h_{i+1}^{\cc}, & \text{if}\;i<n\;\text{and}\;h_{i}=i,\\ i, & \text{if}\;i=n,\;\text{or}\;i<n\;\text{and}\;h_{i}>i.\end{cases}
	\end{displaymath}
\end{corollary}
\begin{proof}
	By definition the least element of $\DD_{n}^{A}$ is the Dyck path $\mathfrak{0}^{A}$ determined by the height sequence $(1,2,\ldots,n)$, and the pseudocomplement of $\pf$ is $\pf\to\mathfrak{0}^{A}$.  Now the result follows by applying Theorem~\ref{thm:rel_pseudocomp_a}. 
\end{proof}

\begin{proposition}\label{prop:regular_a}
	A Dyck path $\pf\in D_{n}^{A}$ is regular if and only if it its height sequence $\hh_{\pf}=(h_{1},h_{2},\ldots,h_{n})$ satisfies either $h_{i}=i$ or if $h_{i}>i$, say $h_{i}=s$, then $h_{i}=h_{i+1}=\cdots=h_{s}=s$.
\end{proposition}
\begin{proof}
	Let $\pf\in D_{n}^{A}$ have height sequence $\hh_{\pf}=(h_{1},h_{2},\ldots,h_{n})$, and let $\hh_{\pf^{\cc}}=\bigl(h_{1}^{\cc},h_{2}^{\cc},\ldots,h_{n}^{\cc}\bigr)$ and $\hh_{(\pf^{\cc})^{\cc}}=\bigl(h_{1}^{\cc\cc},h_{2}^{\cc\cc},\ldots,h_{n}^{\cc\cc}\bigr)$ be the height sequences of $\pf^{\cc}$ and $(\pf^{\cc})^{\cc}$, respectively.

	First suppose that $\pf$ is regular.  If $\pf$ is the least element, its height sequence clearly satisfies the desired conditions.  So suppose that $\pf$ is not the least element.  Hence we can find some $i\in[n-1]$ with $h_{i}>i$.  Corollary~\ref{cor:pseudocomp_a} implies that $h_{i}=h_{i}^{\cc\cc}=h_{i+1}^{\cc\cc}=h_{i+1}$, since $\pf$ is regular.  Since $h_{n}=n$, there must be some minimal index $s>i$ such that $h_{s}=s$, and we find that $h_{i}=h_{i+1}=\cdots=h_{s}=s$ as desired.  

	\smallskip

	Conversely, suppose that $\hh_{\pf}$ has the given properties.  Then either $h_{i}=i$ for all $i\in[n]$, and $\pf$ is the least element and hence regular.  Or there is some $i\in[n-1]$ with $h_{i}>i$, say $h_{i}=s$.  By assumption we have $h_{i}=h_{i+1}=\cdots=h_{s}=s$, and Corollary~\ref{cor:pseudocomp_a} implies $h_{j}^{\cc}=j$ for $j<s$ and $h_{s}^{\cc}=h_{s+1}^{\cc}$.  By using Corollary~\ref{cor:pseudocomp_a} once again, we obtain $h_{i}^{\cc\cc}=h_{i+1}^{\cc\cc}=\cdots=h_{s}^{\cc\cc}=s$ as desired.
\end{proof}

\begin{example}
	The highlighted Dyck paths in Figure~\ref{fig:dominance_a4} are the regular elements of $\DD_{4}^{A}$.  The Dyck path $\pf\in D_{4}^{A}$ given by the height sequence $\hh_{\pf}=(2,3,3,4)$ is for instance not regular, since $\pf^{\cc}$ has height sequence $\hh_{\pf^{\cc}}=(1,2,4,4)$, and $(\pf^{\cc})^{\cc}$ has height sequence $\hh_{(\pf^{\cc})^{\cc}}=(3,3,3,4)$.
\end{example}

A Dyck path $\pf\in D_{n}^{A}$ has a \alert{return at $i$} if the coordinate $(i,i)$ belongs to the path, or equivalently if the prefix of $w_{\pf}$ having length $2i$ contains precisely $i$ times the letter $u$ and $i$ times the letter $r$.  Two returns, say at $i$ and $j$ for $i<j$, are \alert{consecutive} if $\pf$ does not have a return at some $k$ with $i<k<j$.  The returns at $0$ and $n$ are \alert{trivial}.  We can now reformulate Proposition~\ref{prop:regular_a} as follows.

\begin{corollary}\label{cor:regular_paths_a}
	A Dyck path $\pf\in D_{n}^{A}$ is regular if and only if for every pair $(i,j)$ with $0\leq i<j\leq n$ the following is satisfied: if $\pf$ has consecutive returns at $i$ and $j$, then the coordinate $(i,j)$ belongs to the path. 
\end{corollary}
\begin{proof}
	Let $\pf\in D_{n}^{A}$ have height sequence $\hh_{\pf}=(h_{1},h_{2},\ldots,h_{n})$ and suppose that $\pf$ has two consecutive returns at $i$ and $j$ for $i<j$.
	
	First suppose that $\pf$ passes through $(i,j)$.  If $j=i+1$, then we have $h_{i}=i$ and $h_{i+1}=i+1$.  If $j>i+1$, then (since the path passes through $(i,j)$) we have $h_{i}=i$ and $h_{i+1}=h_{i+2}=\cdots=h_{j}=j$.  It follows that $\hh_{\pf}$ satisfies the conditions of Proposition~\ref{prop:regular_a}, and $\pf$ is thus regular.
	
	Conversely, suppose that $\pf$ does not pass through $(i,j)$.  Then we have $h_{i+1}=j'<j$.  Since $i$ and $j$ are consecutive returns it follows that $h_{j'}>j'$, and there must be some minimal $s\in[n]$ with $i+1<s<j'$ and $h_{i+1}=h_{s}<h_{s+1}$.  Again since $i$ and $j$ are consecutive returns, we obtain $h_{i+1}=h_{i+2}=\cdots=h_{s}<s$, which in view of Proposition~\ref{prop:regular_a} implies that $\pf$ is not regular.
\end{proof}

Since the regular elements of a Heyting algebra form a Boolean subalgebra, it is immediately clear that the number of regular elements equals $2^{k}$ for some $k\in\mathbb{N}$.  In view of Corollary~\ref{cor:regular_paths_a} we can be more precise.

\begin{corollary}\label{cor:number_regular_a}
	The number of regular elements of $\DD_{n}^{A}$ is $2^{n-1}$ for $n>0$. 
\end{corollary}
\begin{proof}
	Corollary~\ref{cor:regular_paths_a} implies that for every subset $\{i_{1},i_{2},\ldots,i_{k}\}$ of $[n-1]$ there exists a unique Dyck path $\pf\in D_{n}^{A}$ such that $i_{j}$ is a non-trivial return for $j\in[k]$, and these are precisely the regular paths of $\DD_{n}^{A}$.  This yields the claim.
\end{proof}

\subsection{Type $B$}
	\label{sec:formula_b}
In this section we compute explicit formulas for the relative pseudocomplement and the pseudocomplement in $\DD_{n}^{B}$, and we characterize the regular elements.  In view of Theorem~\ref{thm:dyck_algebras} we can view $\DD_{n}^{B}$ as a Heyting subalgebra of $\DD_{2n}^{A}$, and could in principle derive these formulas from Theorem~\ref{thm:rel_pseudocomp_a}.  However, the formula for the height function of the centrally symmetric Dyck paths in Lemma~\ref{lem:dyck_b_automorphism_height} seems to be rather unhandy, so we prefer a direct computation.  

\begin{theorem}\label{thm:rel_pseudocomp_b}
	Let $\pf_{1},\pf_{2}\in D_{n}^{B}$ with height sequences $\hh_{\pf_{1}}=\bigl(h_{1}^{(1)},h_{2}^{(1)},\ldots,h_{k_{1}}^{(1)}\bigr)$ and $\hh_{\pf_{2}}=\bigl(h_{1}^{(2)},h_{2}^{(2)},\ldots,h_{k_{2}}^{(2)}\bigr)$, respectively.  The relative pseudocomplement $\pf_{1}\to\pf_{2}$ is the Dyck path $\pf\in D_{n}^{B}$ determined by the height sequence $\hh_{\pf}=(h_{1},h_{2},\ldots,h_{k})$ with
	\begin{displaymath}
		k = \begin{cases}k_{2}, & \text{if}\;k_{1}<k_{2},\;\text{or}\;k_{1}\geq k_{2}\;\text{and}\;h_{k_{2}}^{(1)}>h_{k_{2}}^{(2)},\\ \max\bigl\{i\in[k_{2}]\mid h_{i}^{(1)}>h_{i}^{(2)}\bigr\}+1, & \text{if}\;k_{1}\geq k_{2}\;\text{and}\;h_{k_{2}}^{(1)}\leq h_{k_{2}}^{(2)}.\end{cases}
	\end{displaymath}
	If $k_{1}<k_{2}$, then
	\begin{displaymath}
		h_{i} = \begin{cases}h_{i+1}, & \text{if}\;i\leq k_{1}\;\text{and}\;h_{i}^{(1)}\leq h_{i}^{(2)},\\ h_{i}^{(2)}, & \text{if}\;i>k_{1},\;\text{or}\;i\leq k_{1}\;\text{and}\;h_{i}^{(1)}>h_{i}^{(2)},\end{cases}
	\end{displaymath}
	and if $k_{1}\geq k_{2}$, then
	\begin{displaymath}
		h_{i} = \begin{cases}2n-k+1, & \text{if}\;i=k\;\text{and}\;h_{k}^{(1)}\leq h_{k}^{(2)},\\ h_{i+1}, & \text{if}\;i<k\;\text{and}\;h_{i}^{(1)}\leq h_{i}^{(2)},\\ h_{i}^{(2)}, & \text{if}\;i\leq k\;\text{and}\;h_{i}^{(1)}>h_{i}^{(2)}.\end{cases}
	\end{displaymath}
\end{theorem}
\begin{proof}
	Let us first verify that the chosen $k$ is indeed the smallest possible value.  By definition of the dominance order, we observe that the length of the height sequence of $\pf_{1}\wedge_{D}\pf$ must be at least $k_{2}$ in order to satisfy $\pf_{1}\wedge_{D}\pf\leq_{D}\pf_{2}$, and this length is $\max\{k_{1},k\}$.  So, if $k_{1}<k_{2}$, then we have to put $k=k_{2}$.  If $k_{1}\geq k_{2}$, then we can make $k$ as small as possible.  First suppose that $h_{k_{2}}^{(1)}>h_{k_{2}}^{(2)}$.  In view of Lemma~\ref{lem:dyck_b_comparison}, this forces $k_{1}=k_{2}$, and if we choose $k<k_{1}=k_{2}$, then the entry at position $k_{2}$ in the height sequence of $\pf_{1}\wedge_{D}\pf$ is larger than $h_{k_{2}}^{(2)}$, which contradicts $\pf_{1}\wedge_{D}\pf\leq_{D}\pf_{2}$.  Hence the smallest possible value is $k=k_{2}$.  Now suppose that $h_{k_{2}}^{(1)}\leq h_{k_{2}}^{(2)}$.  In order to guarantee that $\pf_{1}\wedge_{D}\pf\leq_{D}\pf_{2}$, we have to choose $k$ at least to be equal to $\max\bigl\{i\in[k_{2}]\mid h_{i}^{(1)}>h_{i}^{(2)}\bigr\}$, and say this value is $j$.  If $h_{j}$ is the last entry of a height sequence of some Dyck path of type $B$, then Lemma~\ref{lem:height_sequence_b} implies $h_{j}\in\{2n-j+1,2n-j\}$.  At the same time, the condition on $\pf$ being a candidate for the relative pseudocomplement forces $h_{j}\leq h_{j}^{(2)}$, and since $j<k_{2}$ Lemma~\ref{lem:height_sequence_b} implies $h_{j}^{(2)}\leq 2n-k_{2}$.  If we put these facts together, we obtain $h_{j}\leq h_{j}^{(2)}\leq 2n-k_{2}<2n-j\leq h_{j}$, which is a contradiction.  We therefore have to choose $k=j+1$, and can then make the corresponding value in the height sequence of $\pf$ as big as possible.

	\smallskip
	
	Now we verify that $\hh_{\pf}$ is indeed the height sequence of some $\pf\in D_{n}^{B}$.  By construction, we have $k\leq k_{2}\leq n$, and either $h_{k}=2n-k+1$ or $h_{k}=h_{k_{2}}^{(2)}$.  Next we show that $h_{i}\leq h_{i+1}$ for $i\in[k]$.  If $h_{i}=h_{i+1}$, we are done, otherwise $h_{i}=h_{i}^{(2)}$.  In this case we need to look at $h_{i+1}$.  If $h_{i+1}=h_{i+2}=\cdots=h_{k}$, then there are two choices, either $h_{k}=h_{k}^{(2)}\geq h_{i}^{(2)}=h_{i}$, and we are done, or $h_{k}=2n-k+1\geq 2n-k_{2}+1\geq h_{k_{2}}^{(2)}\geq h_{i}^{(2)}=h_{i}$ since $k\leq k_{2}$.  If there is some minimal $s<k$ with $h_{s}=h_{s}^{(2)}$, then we obtain $h_{i}=h_{i+1}=\cdots=h_{s}=h_{s}^{(2)}\geq h_{i}^{(2)}=h_{i}$, and we are done.  Lastly we show that $h_{i}\geq i$ for $i\in[k]$.  This is by construction satisfied for $h_{k}$.  So choose $i<k$ maximal such that $h_{i}<i$.  We have two choices, either $h_{i}=h_{i+1}$ and the maximality of $i$ implies $i>h_{i}=h_{i+1}\geq i+1$ which is a contradiction, or $h_{i}=h_{i}^{(2)}$ and the fact that $\pf_{2}$ is a Dyck path of type $B$ implies $i>h_{i}=h_{i}^{(2)}\geq i$, which is a contradiction as well.  Hence Lemma~\ref{lem:height_sequence_b} implies that $\hh_{\pf}$ is indeed the height sequence of some $\pf\in D_{n}^{B}$.  
	
	\smallskip

	Now we show that $\pf$ satisfies $\pf_{1}\wedge_{D}\pf\leq_{D}\pf_{2}$, and let us write $\hh_{\pf_{1}\wedge\pf}=(\bar{h}_{1},\bar{h}_{2},\ldots,\bar{h}_{\bar{k}})$.  We distinguish two cases.
	
	(i) Let $k_{1}<k_{2}$.  In this case we have $\bar{k}=k_{2}$, and $\bar{h}_{i}=h_{i}=h_{i}^{(2)}$ for $i>k_{1}$.  For $i<k_{1}$, we have $\bar{h}_{i}=\min\bigl\{h_{i}^{(1)},h_{i}\bigr\}\leq h_{i}^{(2)}$, and we are done.
	
	(ii) Let $k_{1}\geq k_{2}$.  Since $k\leq k_{2}$, we have $\bar{k}=k_{1}$.  If $i>k$, then by construction $\bar{h}_{i}=h_{i}^{(1)}\leq h_{i}^{(2)}$, and if $i\leq k$, then $\bar{h}_{i}=\min\bigl\{h_{i}^{(1)},h_{i}\bigr\}\leq h_{i}^{(2)}$.  

	 \smallskip

	It remains to show that $\pf$ is indeed the relative pseudocomplement of $\pf_{1}$ with respect to $\pf_{2}$.  Suppose there is some $\pf'\in D_{n}^{B}$ with $\pf_{1}\wedge_{D}\pf'\leq_{D}\pf_{2}$ and $\pf<_{D}\pf'$.  By the choice of $k$, we conclude that $\pf'$ has height sequence $\hh_{\pf'}=(h'_{1},h'_{2},\ldots,h'_{k})$, and by definition of the dominance order there must be a maximal index $i\in[k]$ with $h_{i}<h'_{i}$.  If $i=k$, then by construction we have either $h'_{k}>2n-k+1$, which contradicts Lemma~\ref{lem:height_sequence_b}, or $h'_{k}>h_{k}^{(2)}$, which contradicts the assumption $\pf_{1}\wedge_{D}\pf'\leq_{D}\pf_{2}$.  If $i<k$, then we have two choices again.  Either $h'_{i}>h_{i}=h_{i+1}=h'_{i+1}$ by the maximality of $i$, and this contradicts Lemma~\ref{lem:height_sequence_b}, or $h'_{i}>h_{i}=h_{i}^{(2)}$, which contradicts the assumption $\pf_{1}\wedge_{D}\pf'\leq_{D}\pf_{2}$ again.  This concludes the proof.	
\end{proof}

The following corollary is immediate.

\begin{corollary}\label{cor:pseudocomp_b}
	Let $\pf\in D_{n}^{B}$ with height sequence $\hh_{\pf}=(h_{1},h_{2},\ldots,h_{k})$.  The pseudocomplement of $\pf$ is the Dyck path $\pf^{\cc}\in D_{n}^{B}$ determined by the height sequence $\hh_{\pf^{\cc}}=\bigl(h_{1}^{\cc},h_{2}^{\cc},\ldots,h_{k'}^{\cc}\bigr)$ with
	\begin{displaymath}
		k' = \begin{cases}n, & \text{if}\;k<n,\;\text{or}\;k=n\;\text{and}\;h_{n}=n+1,\\ \max\bigl\{i\in[n]\mid h_{i}>i\bigr\}+1, & \text{if}\;k=n\;\text{and}\;h_{n}=n.\end{cases}
	\end{displaymath}
	If $k<n$, then
	\begin{displaymath}
		h_{i}^{\cc} = \begin{cases}h_{i+1}^{\cc}, & \text{if}\;i\leq k\;\text{and}\;h_{i}=i,\\ i, & \text{if}\;i>k,\;\text{or}\;i\leq k\;\text{and}\;h_{i}>i,\end{cases}
	\end{displaymath}
	and if $k=n$, then
	\begin{displaymath}
		h_{i}^{\cc} = \begin{cases}2n-k'+1, & \text{if}\;i=k'\;\text{and}\;h_{k'}=k',\\ h_{i+1}^{\cc}, & \text{if}\;i<k'\;\text{and}\;h_{i}=i,\\ i, & \text{if}\;i\leq k'\;\text{and}\;h_{i}>i.\end{cases}
	\end{displaymath}
\end{corollary}
\begin{proof}
	By definition, the pseudocomplement of $\pf\in D_{n}^{B}$ is $\pf\to\mathfrak{0}$, where $\mathfrak{0}$ is the least element of $\DD_{n}^{B}$, whose height sequence is $\hh_{\mathfrak{0}}=(1,2,\ldots,n)$.  The result follows by applying Theorem~\ref{thm:rel_pseudocomp_b}.
\end{proof}

\begin{proposition}\label{prop:regular_b}
	A Dyck path $\pf\in D_{n}^{B}$ is regular if and only if its height sequence $\hh_{\pf}=(h_{1},h_{2},\ldots,h_{k})$ satisfies one of the following conditions:
	\begin{enumerate}[(1)]
		\item $h_{k}=n$, and for every $i\in[k-1]$ we have either $h_{i}=i$ or if $h_{i}>i$, say $h_{i}=s$, then $h_{i}=h_{i+1}=\cdots=h_{s}=s$;\quad or  
		\item $h_{k}=2n-k+1,h_{k-1}=k-1$, and for every $i\in[k-2]$ we have either $h_{i}=i$ or if $h_{i}>i$, say $h_{i}=s$, then $h_{i}=h_{i+1}=\cdots=h_{s}=s$.
	\end{enumerate}
\end{proposition}
\begin{proof}
	Let $\pf\in D_{n}^{B}$ have height sequence $\hh_{\pf}=(h_{1},h_{2},\ldots,h_{k})$, and let $\hh_{\pf^{\cc}}=\bigl(h_{1}^{\cc},h_{2}^{\cc},\ldots,h_{k'}^{\cc}\bigr)$ and $\hh_{(\pf^{\cc})^{\cc}}=\bigl(h_{1}^{\cc\cc},h_{2}^{\cc\cc},\ldots,h_{k''}^{\cc\cc}\bigr)$ be the height sequences of $\pf^{\cc}$ and $(\pf^{\cc})^{\cc}$, respectively.

	\smallskip
	
	First suppose that $\pf$ is regular.  We distinguish two cases.
	
	(i)  Let $h_{k}=n$.  This can only happen if $k=n$.  If $h_{i}=i$ for all $i\in[n]$, then $\pf$ is the least element of $\DD_{n}^{B}$, and it clearly satisfies Condition~(1).  Otherwise, there is some $i\in[n-1]$ with $h_{i}>i$.  Since $\pf$ is regular, it follows from Corollary~\ref{cor:pseudocomp_b} that $h_{i}=h_{i}^{\cc\cc}=h_{i+1}^{\cc\cc}=h_{i+1}$.  Since $h_{n}=n$, there must be some minimal index $s\in[n]$ with $s>i$ such that $h_{s}=s$, and we find that $h_{i}=h_{i+1}=\cdots=h_{s}=s$, and $\hh_{\pf}$ thus satisfies Condition~(1).
	
	(ii) Let $h_{k}>n$.  Suppose there is some $i\in[k-2]$ with $h_{i}>i$.  Since $\pf$ is regular, it follows from Corollary~\ref{cor:pseudocomp_b} that $h_{i}=h_{i}^{\cc\cc}=h_{i+1}^{\cc\cc}=h_{i+1}$.  For the moment we claim that $h_{k-1}=k-1$, which implies that there is some minimal index $s\in[k-1]$ with $s>i$ such that $h_{s}=s$, and we find that $h_{i}=h_{i+1}=\cdots=h_{s}=s$.  In order to establish $h_{k-1}=k-1$, we distinguish two more cases:\\
	(iia) Let $k=n$.  Then $h_{n}=n+1$, and Corollary~\ref{cor:pseudocomp_b} implies $k'=n$ and $h_{n}^{\cc}=n$.  Since $\pf$ is regular, we have $k''=k=n$, and thus $h_{n-1}^{\cc}>n-1$, which in view of Corollary~\ref{cor:pseudocomp_b} yields $h_{n-1}^{\cc\cc}=n-1$.  Thus $h_{n-1}=n-1$, since $\pf$ is regular.\\
	(iib) Let $k<n$.  Then $h_{k}=2n-k+1$, and Corollary~\ref{cor:pseudocomp_b} implies $k'=n$ and $h_{s}^{\cc}=s$ for $s\geq k$.  Since $\pf$ is regular, we have $k''=k$, and thus $h_{k-1}^{\cc}>k-1$, which in view of Corollary~\ref{cor:pseudocomp_b} yields $h_{k-1}^{\cc\cc}=k-1$.  Thus $h_{k-1}=k-1$, since $\pf$ is regular.
	
	Hence $\hh_{\pf}$ satisfies Condition~(2).
	
	\smallskip
	
	Conversely, suppose first that $\hh_{\pf}$ has the properties given in Condition~(1).  In view of Lemma~\ref{lem:height_sequence_b}, we conclude $k=n$, and Corollary~\ref{cor:pseudocomp_b} implies $k'\leq n$, and $h_{k'}^{\cc}>n$, which in turn implies $k''=n$ and $h_{n}^{\cc\cc}=n=h_{n}$.  If $i<n$ with $h_{i}=i$, then Corollary~\ref{cor:pseudocomp_b} implies $h_{i}^{\cc}=h_{i+1}^{\cc}\geq i+1$, which in turn implies $h_{i}^{\cc\cc}=i$.  If $i<n$ with $h_{i}>i$, then by assumption if $h_{i}=s$, we have $h_{i}=h_{i+1}=\cdots=h_{s}=s$.  The previous implies that $h_{s}^{\cc\cc}=s$, and Corollary~\ref{cor:pseudocomp_b} implies $h_{t}^{\cc\cc}=h_{t+1}^{\cc\cc}$ for $t\in\{i,i+1,\ldots,s-1\}$.  Hence we have $h_{i}^{\cc\cc}=h_{i+1}^{\cc\cc}=\cdots=h_{c}^{\cc\cc}=c$, and consequently $\hh_{\pf}=\hh_{(\pf^{\cc})^{\cc}}$, which implies that $\pf$ is regular.
	
	Now suppose that $\hh_{\pf}$ has the properties given in Condition~(2).  First say that $k=n$.  Corollary~\ref{cor:pseudocomp_b} implies $k'=n, h_{n}^{\cc}=n$ and $h_{n-1}^{\cc}=h_{n}^{\cc}=n$.  If we apply Corollary~\ref{cor:pseudocomp_b} again, we get $k''=n, h_{n}^{\cc\cc}=n+1=h_{n}$ and $h_{n-1}^{\cc\cc}=n-1=h_{n-1}$.  Now say that $k<n$.  Corollary~\ref{cor:pseudocomp_b} implies $k'=n$, as well as $h_{s}^{\cc}=s$ for $s\in\{k,k+2,\ldots,n\}$ and $h_{k-1}^{\cc}=k$.  If we apply Corollary~\ref{cor:pseudocomp_b} again, we get $k''=k, h_{k}^{\cc}=2n-k+1=h_{k}$ and $h_{k-1}^{\cc\cc}=k-1=h_{k-1}$.  For $i<k-2$, we obtain $h_{i}=h_{i}^{\cc\cc}$ as in the previous paragraph.  Hence we have $\hh_{\pf}=\hh_{(\pf^{\cc})^{\cc}}$, which implies that $\pf$ is regular.
\end{proof}

\begin{example}
	The highlighted Dyck paths in Figure~\ref{fig:dominance_b3} are the regular elements of $\DD_{3}^{B}$.  The Dyck path $\pf\in D_{3}^{B}$ given by the height sequence $\hh_{\pf}=(2,4)$ is for instance not regular, since $\pf^{\cc}$ has height sequence $\hh_{\pf^{\cc}}=(1,2,3)$, and $(\pf^{\cc})^{\cc}$ has height sequence $\hh_{(\pf^{\cc})^{\cc}}=(6)$.
\end{example}

For a Dyck path $\pf\in D_{n}^{B}$ define a \alert{return at $i$} as in the type-$A$ case.  Additionally, $\pf$ has an \alert{upper end at $i$} if the coordinate $(i,2n-i)$ belongs to the path and $0\leq i<n$, or equivalently if $w_{\pf}$ contains the letter $u$ precisely $2n-i$ times.  A return at $i$ and an upper end at $j$ are \alert{consecutive} if $\pf$ does not have a return at some $k$ for $k>i$.  

\begin{corollary}\label{cor:regular_paths_b}
	A Dyck path $\pf\in D_{n}^{B}$ is regular if and only if for every pair $(i,j)$ with $0\leq i<j\leq n$ the following is satisfied: if $\pf$ has consecutive returns at $i$ and $j$, then the coordinate $(i,j)$ belongs to the path, and if $\pf$ has a return at $i$ and an upper end at $j$ which are consecutive, then $i=j<n$.
\end{corollary}
\begin{proof}
	Let $\pf\in D_{n}^{B}$ have height sequence $\hh_{\pf}=(h_{1},h_{2},\ldots,h_{k})$.  
	
	First of all, we assume that $\pf$ has no upper end (and hence it has a return at $n$).  Hence $k=n$ and $h_{n}=n$.  Suppose that $\pf$ has two consecutive returns at $i$ and $j$ for $i<j$.  Say that $\pf$ passes through $(i,j)$.  If $j=i+1$, then we have $h_{i}=i$ and $h_{i+1}=i+1$.  If $j>i+1$, then (since the path passes through $(i,j)$) we have $h_{i}=i$ and $h_{i+1}=h_{i+2}=\cdots h_{j}=j$.  Hence $\hh_{\pf}$ satisfies Condition~(1) of Proposition~\ref{prop:regular_b}, and $\pf$ is thus regular.  Say now that $\pf$ does not pass through $(i,j)$.  Then we have $h_{i+1}=j'<j$.  Since $i$ and $j$ are consecutive returns it follows that $h_{j'}>j'$, and there must be some minimal $s\in[n]$ with $i+1<s<j'$ and $h_{i+1}=h_{s}<h_{s+1}$.  Again since $i$ and $j$ are consecutive returns, we obtain $h_{i+1}=h_{i+2}=\cdots=h_{s}<s$, which in view of Proposition~\ref{prop:regular_b} implies that $\pf$ is not regular.
	
	Now assume that $\pf$ has a return at $i$ and an upper end at $j$, which are consecutive.  If $i=j<n$, then $k=j+1$ and $h_{k}=2n-k+1$ as well as $h_{k-1}=h_{i}=i=k-1$.  In view of the previous paragraph, $\hh_{\pf}$ satisfies Condition~(2) of Proposition~\ref{prop:regular_b}, and $\pf$ is thus regular.  If $i<j<n$, then $k=j$ and $h_{k}=2n-k$ or $h_{k-1}>k-1$, which in view of Proposition~\ref{prop:regular_b} implies that $\pf$ is not regular.
\end{proof}

\begin{corollary}\label{cor:number_regular_b}
	The number of regular elements of $\DD_{n}^{B}$ is $2^{n}$ for $n>0$.
\end{corollary}
\begin{proof}
	First we observe that if $\pf$ does not have an upper end, then $\pf\in D_{n}^{A}$, and in view of Corollary~\ref{cor:number_regular_a} there are $2^{n-1}$ such regular paths.  For any such regular $\pf$ which has non-trivial returns at $i_{1},i_{2},\ldots,i_{k}$, we can create a regular Dyck path $\pf'$ by turning the return at $i_{k}$ into an upper end at $i_{k}$.  In view of Proposition~\ref{prop:regular_b} this is a bijection.  Hence the number of regular elements of $\DD_{n}^{B}$ is $2\cdot 2^{n-1}=2^{n}$ as desired.
\end{proof}

\section*{Acknowledgements}
	\label{sec:acknowledgements}
I would like to thank the anonymous referees for many valuable comments that greatly improved content and presentation of the article.

\bibliography{../../literature}

\end{document}